\documentclass[plain]{imsart}

\RequirePackage{amsthm,amsmath,amsfonts,amssymb}
\RequirePackage[numbers]{natbib}
\RequirePackage{algpseudocode,float,algorithm}
\RequirePackage[numbers]{natbib}
\RequirePackage[colorlinks,citecolor=blue,urlcolor=blue]{hyperref}
\RequirePackage{bm,mathrsfs,makeidx,graphicx,multirow}
\RequirePackage{booktabs}
\RequirePackage{threeparttable}
\RequirePackage{chngpage}
\usepackage{extarrows}
\RequirePackage{xr}
\RequirePackage{tabularx}

\startlocaldefs

\newtheorem{theorem}{Theorem}[section]

\newcommand{\bv}{\mbox{\bf v}}

\newcommand{\bW}{\mbox{\bf W}}

\newtheorem{corollary}{Corollary}
\newtheorem{rmk}{Remark}

\newtheorem{definition}{Definition}

\newtheorem{cond}{Condition}
\newcommand{\non}{\nonumber \\}
\newcommand{\bbA}{{\bf A}}

\newcommand{\bbB}{{\bf B}}

\newcommand{\bbD}{{\bf D}}

\newcommand{\bbe}{{\bf e}}

\newcommand{\bbH}{{\bf H}}

\newcommand{\ep}{\ensuremath{\epsilon}} 

\newcommand{\bbI}{{\bf I}}

\newcommand{\bbM}{{\bf M}}

\newcommand{\bbr}{{\bf r}}
\newcommand{\bbs}{{\bf s}}

\newcommand{\bbV}{{\bf V}}
\newcommand{\bbv}{{\bf v}}

\newcommand{\bbW}{{\bf W}}

\newcommand{\bbX}{{\bf X}}

\newcommand{\bbx}{{\bf x}}

\newcommand{\bbb}{{\bf b}}

\newtheorem{example}{Example}

\newcommand{\E}{\mathbb{E}}

\newcommand{\diag}{\mathrm{diag}}
\newcommand{\rank}{\mathrm{rank}}
\newcommand{\var}{\mathrm{var}}

\newcommand{\bbpi}{{\boldsymbol \pi}}
\newcommand{\bbPi}{{\boldsymbol \Pi}}
\newcommand{\bbTheta}{{\boldsymbol \Theta}}
\newcolumntype{C}[1]{>{\centering\arraybackslash}p{#1}}

\newcommand{\bbzero}{{\bf 0}}
\newcommand{\bbone}{{\bf 1}}

\endlocaldefs

\begin{document}

\begin{frontmatter}
\title{Universal Rank Inference via Residual Subsampling with Application to Large Networks}
\runtitle{Rank Inference}
%\title{A sample article title with some additional note\thanksref{T1}}
%\thankstext{T1}{A sample of additional note to the title.}

\begin{aug}
%%%%%%%%%%%%%%%%%%%%%%%%%%%%%%%%%%%%%%%%%%%%%%%
%% Only one address is permitted per author. %%
%% Only division, organization and e-mail is %%
%% included in the address.                  %%
%% Additional information can be included in %%
%% the Acknowledgments section if necessary. %%
%% ORCID can be inserted by command:         %%
%% \orcid{0000-0000-0000-0000}               %%
%%%%%%%%%%%%%%%%%%%%%%%%%%%%%%%%%%%%%%%%%%%%%%%
\author[A]{\fnms{Xiao}~\snm{Han}\ead[label=e1]{xhan011@ustc.edu.cn}},
\author[A]{\fnms{Qing}~\snm{Yang}\thanksref{s1}\ead[label=e2]{yangq@ustc.edu.cn}}
\and
\author[B]{\fnms{Yingying}~\snm{Fan}\thanksref{s1}\ead[label=e3]{fanyingy@marshall.usc.edu}}
%%%%%%%%%%%%%%%%%%%%%%%%%%%%%%%%%%%%%%%%%%%%%%
%% Addresses                                %%
%%%%%%%%%%%%%%%%%%%%%%%%%%%%%%%%%%%%%%%%%%%%%%
\address[A]{International Institute of Finance, School of Management, University of Science and Technology of China, Hefei, Anhui 230026, China \printead[presep={,\ }]{e1,e2}}

\address[B]{Data Sciences and Operations Department, University of Southern California, Los Angeles, CA 90089 \printead[presep={,\ }]{e3}}
\thankstext{s1}{Yingying Fan and Qing Yang serve as co-corresponding authors. }
\end{aug}

\begin{abstract}
Determining the precise rank is an important problem in many large-scale applications with matrix data exploiting
low-rank plus noise models. In this paper, we suggest a universal approach to rank inference via residual subsampling (RIRS) for testing and estimating rank in a wide family of models, including many popularly used network models such as the degree corrected mixed membership model as a special case. Our procedure constructs a test statistic via subsampling entries of the residual matrix after extracting the spiked components. The test statistic converges in distribution to the standard normal under the null hypothesis, and diverges to infinity with asymptotic probability one under the alternative hypothesis.  The effectiveness of RIRS procedure is justified theoretically, utilizing the asymptotic expansions of eigenvectors and eigenvalues for large random matrices recently developed in \cite{FF18} and \cite{FFH18}. The advantages of the newly suggested procedure are demonstrated through several simulation and real data examples.
\end{abstract}
\end{frontmatter}

\section{Introduction} \label{Sec1}

Matrix data have been popularly encountered in various big data applications.  For example,  many science and social applications involve individuals with complicated interaction systems. Such systems can often be modeled using a network with nodes representing the $n$ individuals and edges representing the connectivity among individuals.  The overall connectivity can thus be recorded in an $n\times n$ adjacency matrix whose entries are zeros and nonzeros, representing the corresponding pair of nodes unconnected or connected, respectively. Examples include the friendship network, the citation network, the predator-prey interaction network, and many others.

There has been a large literature on statistical methods and theory proposed for analyzing matrix data. In the network setting, the observed adjacency matrix is frequently modeled as the summation of a latent deterministic low rank mean matrix and  a random noise matrix, where the former stores all useful information in the data and is often the interest.
One popular assumption is that the rank $K$ of the latent mean matrix is known. However, in practice, such $K$ is generally unknown and needs to be estimated. This paper focuses on estimation and inference on the low rank $K$ in a general model setting including many popularly used network models as special cases.

In our model, the data matrix $\bbX$ can be roughly decomposed as a low rank mean matrix $\bbH$ with $K$ spiked eigenvalues and a noise matrix $\bbW$ whose components are mostly independent.  Here, $K$ is unknown and allowed to slowly diverge with $n$. To infer $K$ with quantified statistical uncertainty, we propose a universal approach for Rank Inference by Residual Subsampling (RIRS). Specifically, we consider the hypothesis test
\begin{equation}\label{eq:hypothesis}
H_0: K=K_0\ \text{ vs. } \ H_1: K>K_0
\end{equation}
with $K_0$ some pre-specified positive integer. The spiked mean matrix with rank $K_0$ can be estimated by eigen decomposition, subtracting which from the observed data matrix yields the residual matrix. Then by appropriately subsampling the entries of the residual matrix, we can construct a test statistic. We prove that under the null hypothesis,  the test statistic converges in distribution to the standard normal, and under the alternative hypothesis, some spiked structure remains in the residual matrix and the constructed test statistic behaves very differently. Thus,  the hypothesis test in  \eqref{eq:hypothesis} can be successfully conducted. Then by sequentially testing the hypothesis \eqref{eq:hypothesis} for $K_0 = 1,\cdots, K_{\max}$ with $K_{\max}$ some large enough positive integer, we can estimate $K$ as the first integer that makes our test fail to reject. We provide theoretical justifications on the effectiveness of our procedure. We show that the size of our test tends to the desired level $\alpha$ as sample size increases,  and establish conditions under which the power approaches one asymptotically. We also show that the sequential procedure correctly estimates the true rank with probability tending to $1-\alpha$ as network size increases.

A key to RIRS's success is the subsampling step.  Although the noise matrix $\bbW$ has mostly independent components, the residual matrix is only an estimate of $\bbW$ and has correlated components.  Intuitively speaking, if too many entries of the residual matrix are sampled, the accumulated estimation error and the correlation among sampled entries would be too large, rendering the asymptotic normality invalid. We provide both theoretical and empirical guidance on how many entries to subsample.
 In the special case where the diagonals of the data matrix $\bbX$ are nonzero independent  random variables (which corresponds to selfloops in network models), a special deterministic sampling scheme can be used  and the RIRS test takes a simpler form.

The structure of low rank mean matrix plus noise matrix is very general and includes many popularly used network models such as the Stochastic Block Model (SBM, \cite{SBM1983, WangWong1987, Abbe2017CommunityDA}), Degree Corrected SBM (DCSBM, \cite{Karrer2011}), Mixed Membership (MM) Model, and Degree Corrected Mixed Membership (DCMM) Model \citep{airoldi2008} as special cases. RIRS procedure and the theory established in this paper are applicable to all these network models and go beyond them.  In network model settings, RIRS can accommodate sparse networks and allows for extreme degree heterogeneity.
Substantial efforts have been made in the literature on estimating $K$ in some specific network models, where $K$ is referred to as the number of communities. For example,  \cite{Mc2013} proposed an MCMC algorithm based on the allocation sampler to cluster the nodes in SBM and simultaneously estimate $K$.  \cite{airoldi2008} developed a general variational inference algorithm to estimate the parameters in MM model with $K$ chosen according to some BIC criterion. \cite{2019J} considered testing \eqref{eq:hypothesis} with $K_0=1$ and  proposed a signed polygon statistic which can accommodate the degree heterogeneity in the DCMM model.
\cite{gao2017} proposed EZ statistics constructed by ``frequencies of three-node subgraphs'' to test \eqref{eq:hypothesis} with $K_0=1$ in the setting of DCSBM. \cite{ba2017} introduced a linear spectral statistic to test
$H_0:K=1$  vs. $H_1: K=2$ under the SBM.  \cite{jin2020estimating}  proposed a stepwise goodness-of-fit test for estimating $K$ under DCSBM. \cite{zhang2020adjusted} suggested an adjusted chi-square test to address the goodness-of-fit testing and model selection problem for DCSBM.
Compared to these works, we consider more general model and general positive integer $K_0$ that can be larger than 1.

 There is also a popular line of work using likelihood based methods to estimate $K$. For example, \cite{Daudin2008}, \cite{Latouche-etal2012}, \cite{saldana2017many}, and \cite{wang2017}, among others.  \cite{chatterjee2015matrix} introduced a universal singular value thresholding procedure for the matrix estimation which can be applied to estimate $K$.  \cite{ChenLei2018} proposed a network cross-validation method for estimating $K$ and proved the consistency of the estimator under  SBM. The cross-validation idea was also explored in \cite{li2020network} under the widely used  inhomogeneous Erd\"{o}s-Renyi model with low rank mean matrix via edge sampling.
 \cite{LL15} proposed to estimate $K$ using the spectral properties of two graph operators --  the
non-backtracking matrix and the Bethe Hessian matrix.
\cite{Zhao7321} proposed to sequentially extract one community at a time by optimizing some extraction criterion, based on which they proposed a hypothesis test for testing the number of communities empirically via permutation method.
\cite{BickelSarkar16} proposed a new test based on the asymptotic distribution of the largest eigenvalue of the appropriately rescaled adjacency matrix for testing whether a network is Erd\"{o}s Renyi or not, and suggested a recursive bipartition algorithm for estimating $K$. \cite{L16} generalized the test in \cite{BickelSarkar16} for testing whether a network is SBM with some specific $K_0$, and proposed a sequential testing idea to estimate the true number of communities.

Among the existing literature reviewed above, the works by \cite{BickelSarkar16} and \cite{L16} are most closely related to ours.
The main idea in both papers is that under the null hypothesis of SBM with $K_0$ communities, the model parameters can be estimated and the residual matrix can be calculated and appropriately rescaled. The rescaled residual matrix is close to a generalized Wigner matrix whose extreme eigenvalues (after recentering and rescaling) converge in distribution to the Tracy-Widom distribution. However, under the alternative hypothesis, the extreme eigenvalues behave very differently.  At a high level, this idea is related to ours in the sense that our proposal is also based on the residual matrix.

RIRS test differs from the literature in the way of using the residual matrix. Instead of investigating the spectral distribution of the residual matrix, we construct RIRS test by subsampling just a fraction of the entries in residual matrix. The subsampling idea ensures that the noise accumulation caused by estimating the mean matrix does not dominate the signal which guarantees the nice performance of our test. Compared to the existing literature, RIRS test behaves more similarly to a nonparameteric test in the sense that there is no assumption on specific structure of the low rank mean matrix. Yet, it is also simple and fast to implement.
Our asymptotic theory is also new to the literature. It is built on the recent developments on random matrix theory in \cite{FF18} and \cite{FFH18}, which establishes the asymptotic expansions of the eigenvalues and eigenvectors for a very general class of random matrices. This powerful result allows us to establish the sampling properties of RIRS test in equally general setting.

The remaining of the paper is organized as follows: Section \ref{Sec2} presents the model setting and motivation for RIRS. We introduce our new approach and establish its asymptotic properties in Section \ref{Sec: RRI}.
Simulations under various models are conducted to justify the performance of RIRS in Section \ref{Sec4}. We further apply RIRS to a real data example in Section \ref{Sec:real}.  Additional simulation examples and all proofs are relegated to the Appendix and the Supplementary material.

\subsection{Notations}
We introduce some notations that will be used throughout the paper.   We use $a\ll b$ to represent $a/b\rightarrow0$ and write $a\lesssim b$ if there exists a positive constant $c$ such that $0\le a\le c b$.
 For a matrix $\bbA$, we use $\lambda_j(\bbA)$ to denote the $j$-th largest eigenvalue, and $\|\bbA\|_F$,  $\|\bbA\|$, and $\|\bbA\|_{\infty}$ to denote the Frobenius norm, the spectral norm, and the maximum elementwise infinity norm, respectively. In addition, denote by $\bbA(k)$ the $k$th row of the matrix $\bbA$.  For a unit vector $\bbx = (x_1,\cdots, x_n)^T$, let $d_x=\|\bbx\|_{\infty}=\max|x_i|$ represent the vector infinity norm.

\section{Model setting and motivation} \label{Sec2}

\subsection{Model setting} \label{Sec2.1}

Consider an $n\times n$ symmetric random matrix $\widetilde\bbX$ admitting the following decomposition
\begin{equation}\label{0908.1}
\widetilde \bbX=\bbH+\bbW,
\end{equation}
where $\bbH = \mathbb E(\widetilde \bbX)$ is the mean matrix with unknown rank $K\ll n$  and $\bbW$ is the noise matrix with bounded and independent entries on and above the diagonals.  
In network applications, the observed matrix $\bbX$ is the adjacency matrix and can be either $\widetilde\bbX$ or $\widetilde\bbX - \diag(\widetilde\bbX)$, with the former  corresponding to network with  selfloops and the latter corresponding to network without selfloops, respectively. An important and interesting question is inferring the unknown rank $K$, which corresponds to the number of communities in network models. We address the problem  by testing the hypotheses \eqref{eq:hypothesis}
under the universal model \eqref{0908.1}.

We note that with some transformation, model \eqref{0908.1} can accommodate  nonsymmetric matrices. In fact, for any matrix $\widetilde \bbX$ that can be written as the summation of a rank $K$ mean matrix and a noise matrix of independent components, we can define a new matrix as
	$$\left(
\begin{array}{ccc}
\bbzero &\ \ \widetilde\bbX\\
\widetilde\bbX^T& \ \ \bbzero\\
\end{array}
\right).$$
It is seen that this new matrix has the same structure as in \eqref{0908.1} with rank $2K$, and our new method and theory both apply. For simplicity of presentation, hereafter we assume the symmetric matrix structure for $\widetilde\bbX$ and $\bbX$.

Write the eigen-decomposition of $\bbH$ as $\bbV\bbD\bbV^T$, where $\bbD=\diag(d_1,...,d_K)$ collects the nonzero eigenvalues of $\bbH$ in decreasing magnitude and $\bbV= (\bbv_1,\cdots, \bbv_K)$ is the matrix collecting the corresponding eigenvectors.  Denote by $\hat d_1, \cdots, \hat d_n$ the eigenvalues of $\bbX$ in decreasing magnitude and $\hat{\bbv}_1,\cdots, \hat{\bbv}_n$ the corresponding eigenvectors. We next discuss the motivation of RIRS.

\subsection{Motivation}\label{sec: motivation}

To gain insights, consider the simple case  when the observed data matrix $\bbX = \widetilde{\bbX}$ and follows model \eqref{0908.1}.  It is seen that $\mathbb E \bbW = \bf 0$.
Intuitively, as $n\rightarrow\infty$, the normalized statistic
$
\sum_{i=1}^nw_{ii}/\sqrt{\sum_{i=1}^n\mathbb{E}w_{ii}^2}
$
converges in distribution to standard normal. Meanwhile,  we expect
$
\sum\limits_{i=1}^n\mathbb{E}w_{ii}^2/\sum\limits_{i=1}^nw_{ii}^2
$
to converge to 1 in probability as $n\rightarrow \infty$.  These two results entail that
\begin{equation}\label{eq:norm-trueW}
\frac{\sum_{i=1}^nw_{ii}}{\sqrt{\sum_{i=1}^nw_{ii}^2}}
\end{equation}
is asymptotically normal as the matrix size $n\rightarrow \infty$.

In the ideal case where the eigenvalues $d_1, \cdots, d_K$ and eigenvectors $\bbv_1$, $\cdots$, $\bbv_K$ are known, a test of the form \eqref{eq:norm-trueW} can be constructed by replacing $w_{ii}$ with $\tilde w_{ii}$ where $\widetilde\bbW  = (\tilde w_{ij})= \bbX - \sum_{k=1}^{K_0}d_k\bbv_k\bbv_k^T$. Under the null hypothesis,  $\widetilde\bbW=\bbW$ and the corresponding test statistic (constructed in the same way as \eqref{eq:norm-trueW}) is asymptotically normal. However, under the alternative hypothesis, $\widetilde\bbW$ still contains some information from the $K-K_0$ smallest spiked eigenvalues and the corresponding eigenvectors and the test statistic is expected to exhibit different asymptotic behavior. Thus, the hypotheses in \eqref{eq:hypothesis} can be successfully tested by using this statistic.

In practice, the eigenvalues and eigenvectors of $\bbH$ are unavailable and need to be estimated. A natural estimate of $\widetilde \bbW$ takes the form
\begin{equation}\label{eq: resid-mat}
\widehat \bbW= (\hat w_{ij})=\bbX-\sum_{k=1}^{K_0}\hat d_k\hat\bbv_k\hat\bbv_k^T.
\end{equation}
 Under $H_0$, the residual matrix $\widehat\bbW$ is expected to be close to $\bbW$, which motivates us to consider test of the form
\begin{equation}\label{s1}
\widetilde T_n = \frac{\sum_{i=1}^n\hat w_{ii}}{\sqrt{\sum_{i=1}^n\hat w^2_{ii}}}.
\end{equation}
Intuitively, the asymptotic behavior of the above statistic should be close to the one in \eqref{eq:norm-trueW}.
Thus, by examining the asymptotic behavior of $\widetilde T_n$ we can test the desired hypotheses.  We will formalize this intuition in a later section.

The statistic in  \eqref{eq:norm-trueW}  is constructed using only the diagonals of $\bbW$.
In theory, the asymptotic normality remains true if we  aggregate any randomly sampled entries of the matrix $\bbW$ (instead of just the diagonals) and normalize properly,  as long as the sampling size is large enough, thanks to the independence of the entries on and above the diagonals of $\bbW$.  However, this does not translate into the asymptotic normality of the test based on $\widehat{\bbW}$
for at least two reasons: First, in applications absence of selfloops, the observed data matrix $\bbX$ takes the form $\widetilde\bbX - \diag(\widetilde\bbX)$ and thus $\widehat{\bbW}$ estimates $\bbW - \diag(\widetilde{\bbX})$ which has nonrandom diagonals. Consequently, test constructed using diagonals of $\widehat \bbW$ becomes invalid. Second,
the entries of $\widehat \bbW$ are all correlated and have errors caused by estimating the corresponding entries of $\bbW$.  Aggregating too many entries of $\widehat \bbW$ will cause too much noise accumulation. This together with the correlations among $\hat w_{ij}$ makes the asymptotic normality of the corresponding test statistic invalid. This heuristic argument is formalized in a later Section \ref{Sec3}.   Thus to overcome these difficulties, we need to carefully choose which and how many entries to aggregate.  These issues are formally addressed in the next section.

\section{Rank inference via residual subsampling} \label{Sec: RRI}
\subsection{A universal RIRS test}

A key ingredient of RIRS is subsampling the entries of $\widehat\bbW$. Specifically, define i.i.d. Bernoulli random variables $Y_{ij}$ with $\mathbb{P}(Y_{ij}=1)=\frac{1}{m}$ for $1\leq i<j\leq n$, where $m$ is some positive integer diverging with  $n$ at a rate that will be specified later. In addition, set $Y_{ji}=Y_{ij}$ for $i<j$. A universal RIRS test that works under the broad model \eqref{0908.1}  takes the following form
\begin{equation}\label{eq: test-general}
T_n=\frac{\sqrt{m}\sum_{i\neq j}\hat w_{ij}Y_{ij}}{\sqrt{2\sum_{i\neq j}\hat w_{ij}^2}}.
\end{equation}
The effect of $m$ is to control on average how many entries of the residual matrix to aggregate for calculating the test statistic. It will be made clear in a moment that $n^2/m$ needs to grow to infinity in order for the central limit theorem to take effect. However, the growth rate cannot be too fast because otherwise the noise accumulation and the correlation in $\hat w_{ij}$ would make the asymptotic normality invalid.

The following conditions will be used in our theoretical analysis.

\begin{cond}\label{cond1}
	$\bbW$ is a symmetric matrix with independent and bounded upper triangular entries (including the diagonals) and $\mathbb{E}w_{ij}=0$ for $i\neq j$.
\end{cond}

\begin{cond}\label{cond2}
For  $1\le i<j\le K$, if $|d_i|\neq |d_j|$, there exists a  positive constant $c_0$ such that $\frac{|d_i|}{|d_j|}\ge 1+c_0$.
\end{cond}

\begin{cond}\label{cond3}
	There exists a positive sequence $\theta_n$, which may converge to $0$ as $n\rightarrow \infty$, such that $\sigma_{ij}^2=\var(w_{ij})\le \theta_n$ and $\max\limits_{1\le i\le n}|h_{ii}|\lesssim\theta_n$.
	In addition, $\max\{n\theta_n,\log n\}\lesssim\alpha_n^2:=\max\limits_{i}\sum\limits_{j=1}^n\sigma_{ij}^2$,  $|d_K|\gtrsim \frac{\alpha_n^2}{(\log \log n)^{\ep'}}$ and {\color{black}$\frac{|d_K|}{\alpha_n}\gtrsim  (\log n)^{1+\ep}$} for some positive constants $\ep$, $\ep'$.

\end{cond}
\begin{cond}\label{cond4}
	$\|\bbV\|_{\infty}\lesssim\frac{1}{\sqrt n}$.
\end{cond}

\begin{cond}\label{cond5}
It holds that $\sum_{i\neq j} \sigma_{ij}^2\gg m$. In addition, for some small positive constant $\ep_1<\ep$ with $\ep$ the constant in Condition \ref{cond3},
$$\sum_{i\neq j}\sigma_{ij}^2\gtrsim  (\log n)^{\ep_1}\Big(\frac{n\sum_{k=1}^{K}(\mathbf{1}^T\bbv_k)^2}{m} + \alpha_n^2(\log n)^2+\frac{n}{m} + \frac{n^2\alpha_n^2(\log n)^6}{md_K^2}\Big).$$
\end{cond}

\begin{cond}\label{cond6}
The true rank satisfies that $K\leq O(\log\log n)$.
\end{cond}

With the above conditions, Theorem \ref{thma} below provides the asymptotic null distribution of RIRS test and Theorem \ref{thmb} establishes the asymptotic alternative distribution.

\begin{theorem}\label{thma}
	Assume Conditions \ref{cond1}-\ref{cond6}. 
	Under null hypothesis in \eqref{eq:hypothesis} we have
	\begin{equation}\label{0307.3}
	T_{n}\stackrel{d}{\rightarrow } N(0,1), \text{ as } n\rightarrow\infty.
	\end{equation}
\end{theorem}

\begin{theorem}\label{thmb}
	Assume Conditions \ref{cond1}-\ref{cond6} and the alternative hypothesis in \eqref{eq:hypothesis}.
If $$\sum_{i\neq j}\big(\sum_{k=K_0+1}^{K}d_k\bbv_k(i)\bbv_k(j)\big)^2\ll \sum_{i\neq j}\sigma_{ij}^2,$$
then as  $n\rightarrow \infty$,

 		\begin{equation}\label{eq: alt-distr}
	T_n - \frac{\sqrt{m}\sum_{i\neq j}\sum_{k=K_0+1}^{K}d_k\bbv_k(i)\bbv_k(j)Y_{ij}}{\sqrt{2\sum_{i\neq j}\hat w^2_{ij}}}\stackrel{d}{\rightarrow } N(0,1).
 	\end{equation}
In addition, if
\begin{equation}\label{eq: m-cond2}
\left|\sum_{k=K_0+1}^{K}d_k\sum_{i\neq j}\bbv_k(i)\bbv_k(j)\right|\gg \sqrt{m}\left(\sqrt{\sum_{i\neq j}\mathbb \sigma_{ij}^2}+\sqrt{K-K_0}\sum_{k=K_0+1}^{K}|d_k|\right),
\end{equation}
we have
\begin{equation}\label{0307.5}
\mathbb P(	|T_n| > C) \rightarrow 1, \text{ as } n\rightarrow\infty
	\end{equation}
	for arbitrarily large positive constant $C$.
\end{theorem}

The result in \eqref{eq: alt-distr} guarantees  that if $\sum_{k=K_0+1}^{K}d_k\sum_{i\neq j}\bbv_k(i)\bbv_k(j)Y_{ij}$ is non-negligible compared with $\sqrt{2m^{-1}\sum_{i\neq j}\hat w_{ij}^2}$, then our test has non-vanishing power.  If in addition, the asymptotic mean is large enough such that \eqref{eq: m-cond2} is satisfied, then the asymptotic power can reach one.  The results on asymptotic size and power of RIRS test are formally summarized in the following Corollary.

\begin{corollary}\label{co1}
Under the conditions of Theorem \ref{thma}, we have
$$\lim_{n\rightarrow \infty}\mathbb{P}(|T_n|\geq \Phi^{-1}(1-\alpha/2)|H_0)=\alpha,$$
where $\Phi^{-1}(t)$ is the inverse of the standard normal distribution function, and $\alpha$ is the pre-specified significance level. Alternatively under the same conditions for ensuring \eqref{0307.5}, we have
$$\lim_{n\rightarrow \infty}\mathbb{P}(|T_n|\geq\Phi^{-1}(1-\alpha/2)|H_1)=1.$$
\end{corollary}

\subsection{Remarks on the conditions}\label{rmk:cond}

Random matrix satisfying Condition \ref{cond1} is often termed as generalized Wigner matrix in the literature.  Condition \ref{cond2} allows for eigenvalue multiplicity and requires that there is enough gap between distinct eigenvalues. The constant $c_0$ can be replaced with some slowly vanishing term such as $(\log n)^{-1}$ and our main results will still hold with relevant conditions updated accordingly.

Condition  \ref{cond3} constraints that the nonzero eigenvalues of the low rank mean matrix should have enough spikiness. The two constraints $|d_K|\gtrsim \alpha_n^2(\log \log n)^{-\ep'}$ and $|d_K|\gtrsim \alpha_n(\log n)^{1+\ep}$ in Condition \ref{cond3} are imposed for controlling the noise accumulation in our test caused by estimating $w_{ij}$. The second constraint is a signal strength condition commonly imposed in random matrix theory literature; see \cite{BY08}, \cite{bao2018singular}  and \cite{P07},  among others. The logarithmic factor, $(\log n)^{1+\ep}$, measures the gap between the signal (spiked) eigenvalue and the noise eigenvalue, and is hard to be removed completely because otherwise the sample eigenvector would depend on the noise matrix $\bW$ in a complicated way that is not useful for statistical inference. The first constraint can be satisfied by many network models with low rank structure.   To see this, note that if $X_{ij}$, $j\ge i\ge1$ follows Bernoulli distribution and  $h_{ij}\sim \theta_n$ with $\max_{i,j}h_{ij} <1$, then $\alpha_n^2\sim n
\theta_n$. Since $h_{ij}$'s and $\sigma_{ij}^2$'s are the means and variances of Bernoulli random variables, respectively, we have $h_{ij}\sim \sigma_{ij}^2\sim \theta_n$ and $\|\bbH\|_F = \{\sum_{i,j}h_{ij}^2\}^{1/2} \sim n\theta_n$. Note also that $\|\bbH\|_F = \{\sum_{i=1}^Kd_i^2\}^{1/2}$. Assuming that $K$ is finite and $d_1\sim d_K$, these results together  with $\alpha_n^2\sim n
\theta_n$ derived earlier ensure that $|d_K|\gtrsim \alpha_n^2(\log \log n)^{-\ep'}$ is satisfied.

Condition \ref{cond5} characterizes what kind of $m$ can make RIRS succeed. More detailed discussion on the choice of $m$ will be given in  Section \ref{sec:choice-m}. Condition \ref{cond6} allows the rank $K$ to grow with network size $n$.

Condition \ref{cond4} is a technical condition needed for proving key Lemmas 2-3.  We remark that it can hold under extreme degree heterogeneity in network models.  The following example is used to illustrate this point.

\begin{example}
Consider DCSBM with $K=2$ where mean matrix takes the form
\begin{equation}\label{0105.1}
\bbH=\bbTheta\bbPi\bbB\bbPi^T\bbTheta.
\end{equation}
Here, $\bbB$ is a $2\times 2$ nonsingular matrix with diagonals 1 and off diagonals  taking a constant value in $[0,1)$, $\bbTheta$ is a diagonal matrix with the first $n/2$ diagonal entries taking the same value $\vartheta_1>0$ and the remaining diagonal entries taking the same value $\vartheta_2>0$,  and $\bbPi \in \mathbb{R}^{n\times 2}$ has the first $n/2$ rows equal to $(1,0)$ and the remaining ones equal to $(0,1)$.  Here, for simplicity we assume $n$ is an even number.  It is seen that the first $n/2$ nodes belong to community 1 and share the common degree parameter $\vartheta_1$,  and  the remaining belong to community 2 and share the common  degree parameter $\vartheta_2$.  Since the population eigenvector $\bbv_k$, $k=1,2$ satisfies
$$
\bbTheta\bbPi\bbB\bbPi^T\bbTheta\bbv_k = \lambda_k \bbv_k,
$$
we see that $\bbv_k$ takes the form $(a_1\mathbf{1}_{n/2}^T, a_2\mathbf{1}_{n/2}^T)^T$ with $\mathbf{1}_{n/2}\in\mathbb{R}^{n/2}$ a vectors of 1's and $a_1$ and $a_2$ two constants.  Since $\|\bbv_k\|_2=1$, it follows that $\max\{|a_1|, |a_2|\}\lesssim \frac{1}{\sqrt{n}}$ and thus Condition \ref{cond4} holds regardless of the values of $\vartheta_1$ and $\vartheta_2$.
 \end{example}

\subsection{Choice of $m$}\label{sec:choice-m}
It is seen from the previous two theorems that the tuning parameter $m$ plays a crucial role for RIRS to achieve the desired size with power tending to one.  Condition \ref{cond5} provides general conditions on the choice of $m$ for ensuring the null and alternative distributions in \eqref{0307.3} and \eqref{eq: alt-distr}.  For \eqref{0307.5} to hold, we also need the additional assumption \eqref{eq: m-cond2}. In some special cases, these conditions boil down to simpler forms which can provide us more specific guideline on the choice of $m$.

 As an example, we consider the special case
 \begin{align}\label{eq: cond-choose-m}
& (\log n)^{\delta} \min_{i\neq j}\sigma_{ij}^2\gtrsim   \max_{i\neq j}\sigma_{ij}^2,\non
  \text{ for } K_0<K,  \  \sum_{i\neq j}\sigma_{ij}^2 \lesssim & \left|\sum_{k=K_0+1}^{K}d_k\sum_{i\neq j}\bbv_k(i)\bbv_k(j)\right|, \text{ and }
  \ |d_1|\lesssim n\theta_n,
 \end{align}
where $\delta<\ep$ is some small positive constant with $\ep$ the same as in Condition \ref{cond3}.
The first condition above is for guaranteeing  both the asymptotic size and power results, while the remaining two conditions will be only used for verifying \eqref{eq: m-cond2} which is only needed in establishing the power results.  In network models,  the $(\log n)^{\delta}$ factor  in the first condition of \eqref{eq: cond-choose-m}  is related to degree heterogeneity. For very extreme degree heterogeneity, we may have $\max_{i\neq j}\sigma_{ij}^2/\min_{i\neq j}\sigma_{ij}^2$ diverge at some polynomial rate of sample size. We remark that similar results can be derived using identical proof idea for Theorem \ref{thm: m-choice} with some technical conditions appropriately modified.

We next discuss some more specific network models which can give us more insights on the three conditions in \eqref{eq: cond-choose-m}.
Consider the same DCSBM in \eqref{0105.1} except that $\bbTheta = \text{diag}\{\vartheta_1,\cdots, \vartheta_n\}$ with $\vartheta_j>0$, $j=1,\cdots, n$ the degree parameters.
Assume $\max_{i,j}\{\vartheta_j\vartheta_jb_{\mathcal{C}_i\mathcal{C}_j}\}$ is bounded away from 1 by some constant, where $\mathcal{C}_i \in \{1, 2\}$ is the membership for node $i$,  and $b_{kl}\in (0,1]$ is the $(k,l)$ entry of matrix $\bbB$ taking constant values.  Then we have $\sigma_{ij}^2\sim h_{ij} = \vartheta_i\vartheta_jb_{\mathcal{C}_i\mathcal{C}_j}$ because the entries of $\bbX$ have Bernoulli distributions.     It is seen that the ratio $\max_{i\neq j} \sigma^2_{ij}/\min_{i\neq j}\sigma_{ij}^2 \sim \max_{i\neq j}(\vartheta_i\vartheta_j)/ \min_{i\neq j}(\vartheta_i\vartheta_j)$. Thus, the first condition in \eqref{eq: cond-choose-m} is satisfied if $\max_j\vartheta_j/\min_j\vartheta_j \lesssim \sqrt{(\log n)^{\delta}}$.  {\color{black} It is also straightforward to see that $\theta_n=\max_{i,j}\{\vartheta_i\vartheta_jb_{\mathcal{C}_i\mathcal{C}_j}\}$ and $|d_1|\sim \|\bbH\|_F\lesssim n\max_{i,j}\{\vartheta_i\vartheta_jb_{\mathcal{C}_i\mathcal{C}_j}\}$.  Thus, the last condition $|d_1|\lesssim n\theta_n$ in \eqref{eq: cond-choose-m} holds.  The second condition reduces to $\sum_{i\neq j}\sigma_{ij}^2 \lesssim  \left|d_2\sum_{i\neq j}\bbv_2(i)\bbv_2(j)\right|$ in this model setting.  Since $\sum_{i\neq j}\sigma_{ij}^2 \leq n^2\max_{i\neq j}\{\vartheta_i\vartheta_jb_{\mathcal{C}_i\mathcal{C}_j}\}$, the second condition is satisfied if $n^2 \max_{i\neq j}\vartheta_i\vartheta_jb_{\mathcal{C}_i\mathcal{C}_j} \lesssim  \left|d_2\sum_{i\neq j}\bbv_2(i)\bbv_2(j)\right|$.
}

The next theorem specifies what kind of $m$ satisfies the two inequalities in Condition \ref{cond5} and \eqref{eq: m-cond2}.

\begin{theorem}\label{thm: m-choice}
	Set $\theta_n = \max_{i\neq j}\sigma_{ij}^2$.  Assume \eqref{eq: cond-choose-m} and Condition \ref{cond3},  and let $\ep_1\in (\delta,\ep)$ be some small constant. Then $m$ satisfying the following condition
	
{\color{black}
\begin{equation}\label{eq:m-cond1}
\theta_n^{-1}(\log n)^{\delta+\ep_1}K+n^{-1}\theta_n^{-2}(\log n)^{6+2\delta+\ep_1}(\log\log n)^{2\ep'}\ll m\ll (n/K)^2\theta_n (\log n)^{-2\delta}
	\end{equation}
}
makes Condition \ref{cond5} and inequality \eqref{eq: m-cond2} hold. Consequently, \eqref{0307.3} and \eqref{0307.5} hold under Conditions  \ref{cond1}--\ref{cond4} and \ref{cond6}. Moreover, a sufficient condition for \eqref{eq:m-cond1} is {\color{black}$n(\log n)^{-2\ep+\ep_1+2}(\log\log n)^{-2\ep'}K\ll m \ll nK^{-2}(\log n)^{2\ep-\delta+2}(\log\log n)^{2\ep'}$} under Conditions \ref{cond1}--\ref{cond4}.
\end{theorem}

It is seen that Theorem \ref{thm: m-choice} allows for a wide range of values for $m$. In theory, any $m$ satisfying \eqref{eq:m-cond1} guarantees the asymptotic size and power of our test. In implementation, we found smaller $m$ in this range yields better empirical size.
It is also seen from \eqref{eq:m-cond1} that RIRS works with sparse networks. In fact, the only sparsity condition imposed by \eqref{eq:m-cond1}  is that {\color{black}$\theta_n \gg n^{-1}K^{3/2}(\log n)^{2+\delta+\ep_1/2}$}.
Our sparsity condition is  weaker than many existing ones in related work in the literature.  In particular, both \cite{BickelSarkar16} and \cite{L16} considered dense SBM with $\theta_n$ bounded below by some constant.

We remark that sparse models have been considered in the network literature, though mostly in estimation instead of inference problems. For example, \cite{wang2017} proposed a model selection criterion for estimating $K$ under the sparse setting of SBM with $n\theta_n/\log n \rightarrow \infty$. \cite{LL15} established the consistency of  their method for estimating $K$ under the setting $n\theta_n = O(1)$. We consider the statistical inference problem of hypothesis testing, which involves more delicate analyses for establishing the asymptotic distributions of the test statistic.

\subsection{A special case: networks with selfloops}

We formalize the heuristic arguments in Section \ref{sec: motivation} about the ratio statistic $\widetilde T_n$ in \eqref{s1} when the network admits selfloops. In such a case, the general test \eqref{eq: test-general} still works. However, the simpler one $\widetilde T_n$ can enjoy similar asymptotic properties without the trouble of choosing $m$.

\begin{theorem}\label{thm2}
	Suppose that Conditions \ref{cond1}-\ref{cond4} and \ref{cond6} hold, the network contains selfloops and {\color{black}$\sqrt{\sum_{i=1}^n\sigma_{ii}^2}\gg (\log n)^{2+\ep_2}$ for some positive constant $\ep_2$}.
	
	(i) Under null hypothesis we have
	\begin{equation}\label{ti1}
\widetilde T_n \stackrel{d}{\rightarrow } N(0,1), \text{ as } n\rightarrow \infty.
	\end{equation}

(ii) Under alternative hypothesis, 	if further $\sum_{i=1}^n(\sum_{k=K_0+1}^{K}d_kv^2_k(i))^2\ll \sum_{i=1}^n\sigma_{ii}^2$, we have
	\begin{equation}\label{ti2}
\widetilde T_n  -	\frac{\sum_{k=K_0+1}^{K}d_k}{\sqrt{\sum_{i=1}^n\hat w^2_{ii}}}\stackrel{d}{\rightarrow } N(0,1), \text{ as } n\rightarrow\infty.
	\end{equation}
In addition, if $|\sum_{k=K_0+1}^{K}d_k|^2\gg \sum_{i=1}^n\sigma^2_{ii}+\sum_{i=1}^n(\sum_{k=K_0+1}^{K}d_kv^2_k(i))^2$, then
	\begin{equation}\label{0614.1}
\mathbb P(	|\widetilde{T}_n|>C) \rightarrow 1,
	\end{equation}
for arbitrarily large positive constant $C$.
\end{theorem}

It is seen that with the same critical value $\Phi^{-1}(1-\alpha/2)$, $\widetilde T_n$ enjoys the same asymptotic properties on size and power as $T_n$. In addition, since the construction of $\widetilde T_n$ does not depend on any tuning parameter, the implementation is much easier.

\subsection{Estimation of $K$}\label{sec: K-est}
RIRS naturally suggests a simple method for estimating the rank $K$. The idea is similar to the one in \cite{L16}. That is,
we sequentially test the following hypotheses
 $$H_0: K=K_0 \quad \text{vs.} \quad H_1:K>K_0,$$
for $K_0=1,2,...,K_{\max}$ at a given significance level $\alpha\in (0,1)$ using RIRS.
 Here, $K_{\max}$ is some prespecified positive integer that should be larger than the true rank. In the application, we may select $K_{\max}=\lfloor C\log\log n\rfloor$ with some positive constant $C$. Once RIRS fails to reject a value of $K_0$, we stop and use it as the estimate of the rank. Denote by $\widehat K$ our resulting estimate.

\begin{corollary}\label{coro2}
Suppose there exists some positive constant $\delta_1$ such that
\begin{equation}\label{eq:estim}
\left|\sum_{k=K_0+1}^{K}d_k\sum_{i\neq j}\bbv_k(i)\bbv_k(j)\right|\gg \sqrt{m\log^{\delta_1} n}\left(\sqrt{\sum_{i\neq j}\mathbb \sigma_{ij}^2}+\sqrt{K-K_0}\sum_{k=K_0+1}^{K}|d_k|\right),
\end{equation}
holds for all $K_0<K$. Under the conditions of Theorem \ref{thmb}, we have
\[
\mathbb P(\widehat{K}=K)\rightarrow 1-\alpha.
\]
\end{corollary}

\begin{rmk}
  The additional condition \eqref{eq:estim} ensures that the estimation procedure rejects all $K_0<K$ with asymptotic probability 1 even when $K$ diverges.
\end{rmk}

\begin{rmk}\label{rmk:est} Our numerical studies suggest that the RIRS based sequential testing approach may underestimate $K$ when the network is very sparse or the signal strength is very low; see simulation results in the Supplementary material. To improve the estimation accuracy, we propose to add a penalty on underestimation to $|T_n|$ as follows
\begin{equation}\label{eq:estim2}
|\check{T}_n| = |T_n|+\frac{|\max_i\sum_{j}X_{ij}|^{1/2}}{|\hat d_{K_0}|-|\hat d_{K_0+1}|}.
\end{equation}
Then we implement our estimation procedure identically with $|T_n|$ replaced with $|\check{T}_n|$.  It is easy to show that Corollary \ref{coro2} still holds.
 The intuition is that the penalty term in \eqref{eq:estim2} tends to zero with asymptotic probability one when $K_0=K$, and provides a positive penalty when $K_0<K$. This modified test is useful for increasing the estimation accuracy, especially when $K$ is large, as demonstrated in Table S.3 in the Supplementary file.
\end{rmk}

\subsection{Networks without selfloops: why subsampling?}\label{Sec3}
In this section, we formalize the heuristic arguments given in Section \ref{sec: motivation} on why sub-sampling is necessary. We theoretically show that $T_n$ is no longer asymptotically normal under $H_0$ without the ingredient of subsampling.  For technical simplicity,  we constraint ourselves to the setting of finite $K$ and no eigenvalue multiplicity in this subsection.

We start with introducing some additional notations that will be used in this subsection. For any matrices $\bbM_1$ and $\bbM_2$ of appropriate dimensions, let
\begin{equation}\label{defr}
\mathcal{R}(\bbM_1,\bbM_2,t)=-\sum_{l=0,l\neq 1}^L\frac{\bbM_1^T\mathbb{E}\bbW^l\bbM_2}{t^{l+1}},
\end{equation}
 where $L=\lfloor\log n\rfloor$.

 By Lemma 6 and Theorem 1 of \cite{FF18}, for each $k=1,\cdots, K$,  there exists a unique deterministic $t_k$ such that $\frac{t_k}{d_k}\rightarrow 1$ as $n\rightarrow\infty$ and $\hat d_k-t_k=\bbv_k^T\bbW\bbv_k+O_p(\frac{\alpha_n}{|d_k|})$.
 Define
 \begin{align*}
&\bbb^T_{\bbe_i,k,t}=\bbe_i^T-\mathcal{R}(\bbe_i,\bbV_{-k},t)\Big((\bbD_{-k})^{-1}+\mathcal{R}(\bbV_{-k},\bbV_{-k},t)\Big)^{-1}\bbV_{-k}^T,\\
&\bbs_{k,i}=\bbb_{\bbe_i,k,t_k}-\bbe_i^T\bbv_k\bbv_k, \quad \bbs_k=\sum_{i=1}^n\bbs_{k,i},\ \bbs_k(i)=\bbe_i^T\bbs_k,
\\
&\text{and } \quad \bbr_k=\bbV_{-k}(t_k\bbD_{-k}^{-1}-\bbI)^{-1}\bbV_{-k}^T\mathbb{E}\bbW^2\bbv_k,
 \end{align*}
  where $\bbV_{-k}$ is the submatrix of $\bbV$ by removing the $k$-th column, and we slightly abuse the notation and use $\bbD_{-k}$ to denote the submatrix of $\bbD$ by removing the $k$th diagonal entry.

 Further define $a_k=\sum_{i=1}^n\bbv_k(i)$, $k=1,\cdots, K$ and
 \begin{align}
R(K)&=2\sum_{k=1}^{K}\frac{\mathbf{1}^T\mathbb{E}\bbW^2\bbv_ka_k}{t_k}-2\sum_{k=1}^{K}\frac{a_k^2\bbv_k^T\mathbb{E}\bbW^2\bbv_k}{d_k}\\
&+\sum_{k=1}^{K}\bbv_k^T\diag(\bbW)\bbv_ka_k^2+2\sum_{k=1}^{K}a_k\bbs_k^T\diag(\bbW)\bbv_k+2\sum_{k=1}^{K}a_k\frac{\mathbf{1}^T\bbr_k}{t_k}.\nonumber
\end{align}
We have the following theorems.

\begin{theorem}\label{thm1}
	Suppose that Conditions \ref{cond1}--\ref{cond4} hold with no eigenvalue multiplicity,  and
	{\small\begin{equation}\label{eq: cond-var} \sum_{i<j}\sigma^2_{ij}\left(1-\sum_{k=1}^{K_0}a_k^2\bbv_k(i)\bbv_k(j)-\sum_{k=1}^{K_0}a_k\Big(\bbv_k(j)\bbs_k(i)+\bbv_k(i)\bbs_k(j)\Big)\right)^2\ge {\color{black}(\log n)^{6+\ep_1}\Big(n+\frac{n^2\alpha_n^2}{d_{K_0}^2}\Big)},
	\end{equation}}
 for some positive constant $\ep_1$.
 Under null hypothesis, as $n\rightarrow\infty$, we have

{\small	$$\frac{\sum\limits_{i\neq j}\hat w_{ij}+R(K_0)}{2\sqrt{\sum\limits_{i<j}\sigma^2_{ij}\left(1-\sum\limits_{k=1}^{K_0}a_k^2\bbv_k(i)\bbv_k(j)-\sum\limits_{k=1}^{K_0}a_k\big(\bbv_k(j)\bbs_k(i)+\bbv_k(i)\bbs_k(j)\big)\right)^2}}\stackrel{d}{\rightarrow } N(0,1).$$}
\end{theorem}

\begin{theorem}\label{thm3}
Suppose that Conditions \ref{cond1}--\ref{cond4} hold with  no eigenvalue multiplicity.  In addition, assume
\eqref{eq: cond-var} holds with $K_0$ and $d_{K_0}$ replaced with $K$ and $d_K$, respectively. Under alternative hypothesis, as $n\rightarrow\infty$, we have
{\small\begin{align*}
\frac{\sum\limits_{i\neq j}\hat w_{ij}+R(K)-\sum\limits_{k=K_0+1}^{K}d_ka_k^2}{2\sqrt{\sum\limits_{i<j}\sigma^2_{ij}\left(1-\sum\limits_{k=1}^{K}a_k^2\bbv_k(i)\bbv_k(j)-\sum\limits_{k=1}^{K}a_k\Big(\bbv_k(j)\bbs_k(i)+\bbv_k(i)\bbs_k(j)\Big)\right)^2}}
\stackrel{d}{\rightarrow } N(0,1).
\end{align*}
}
\end{theorem}

It is seen from Theorems \ref{thm1} and \ref{thm3} that aggregating all entries of the residual matrix leads to a statistic with  bias and variance taking very complicated forms under both null and alternative hypotheses. The   complicated forms of bias and variance limit the practical usage of the above results. In addition, and more importantly,  these asymptotic normality results may even fail to hold in some cases.
To understand this, note that the  variance of $\sum_{i\neq j}\hat w_{ij}+R(K_0)$ in Theorem \ref{thm1}  is approximately equal to $$4\sum_{i<j}\sigma^2_{ij}\left(1-\sum_{k=1}^{K_0}a_k^2\bbv_k(i)\bbv_k(j)-\sum_{k=1}^{K_0}a_k(\bbv_k(j)\bbs_k(i)+\bbv_k(i)\bbs_k(j))\right)^2.$$
Condition \eqref{eq: cond-var} is imposed to put a lower bound on the variance. Without this condition, the asymptotic normality in Theorem \ref{thm1} will no longer hold.  We next give an example where inequality \eqref{eq: cond-var}  and the asymptotic normality both fail to hold.  This justifies the necessity of the subsampling step.

\begin{example}
Consider networks with eigenvector taking the form $\bbv_1 = \frac{1}{\sqrt n}\mathbf{1}$.  Then $a_1=\sqrt n$. Since $\bbv_k$, $k\ge 2$ are orthogonal to $\bbv_1$, we have $a_k=0$, $k\ge 2$.  By Condition \ref{cond4} and {Theorem S.1} in the Supplementary file, we have
$\max_i|\bbs_1(i)|\lesssim \frac{\alpha_n^2}{\sqrt n d_1^2}$. Combining this with Condition \ref{cond3} and using the fact $\bbv_1 = \frac{1}{\sqrt n}\mathbf{1}$, we have
\begin{align}
&\sum_{i<j}\sigma^2_{ij}\left(1-\sum_{k=1}^{K_0}a_k^2\bbv_k(i)\bbv_k(j)-\sum_{k=1}^{K_0}a_k(\bbv_k(j)\bbs_k(i)+\bbv_k(i)\bbs_k(j))\right)^2\non
&= \sum_{i<j}\sigma^2_{ij}\left(1-n\bbv_1(i)\bbv_1(j)-\sqrt{n}\big(\bbv_1(j)\bbs_1(i)+\bbv_1(i)\bbs_1(j)\big)\right)^2\non
&= \sum_{i<j}\sigma^2_{ij}\left(\bbs_1(i)+\bbs_1(j)\right)^2\non
&\lesssim \frac{\alpha_n^4}{nd_1^4}\sum_{i<j}\sigma^2_{ij}\le \frac{\alpha_n^6}{d_1^4}\lesssim \frac{n^2\alpha_n^2}{d_1^4}\lesssim \Big(1+\frac{n^2\alpha_n^2}{d_{K_0}^2}\Big),\nonumber
\end{align}
where in the last line we have used  $\sum_{i<j}\sigma_{ij}^2\leq n\alpha_n^2$ and $\alpha_n^2\lesssim n$.
This contradicts \eqref{eq: cond-var}.  
Further,  by checking the proof of Theorem \ref{thm1}, we see that the intrinsic  problem is when aggregating too many terms from the residual matrix, the noise accumulation is no longer negligible, canceling the first order term $\sum_{i\neq j}\sigma_{ij}^2$, and consequentially makes the central limit theorem fail.   Similar phenomenon happens under the alternative hypothesis as well.
\end{example}

\section{Simulation studies} \label{Sec4}

In this section, we use simulations to justify the performance of RIRS in testing and estimating $K$. Section \ref{Sec4-1} considers the network model and Section \ref{Sec4-2} considers more general low rank plus noise matrices. The nominal level is fixed to be $\alpha = 0.05$ in all settings.

\subsection{ Network models}\label{Sec4-1}
Consider the DCMM model \eqref{0105.1}.
We simulate two types of nodes: pure node with $\bbpi_i$ chosen from the set of unit vectors
\begin{equation*}
\text{PN}(K)=\{\bbe_1,\cdots,\bbe_K\},
\end{equation*}
and the mixed membership node with $\bbpi_i$ chosen from
\begin{equation*}
\text{MM}(K,x)=\Big\{(x,1-x,\underset{K-2}{\underbrace{0,\cdots,0}} ),\quad (1-x,x,\underset{K-2}{\underbrace{0,\cdots,0}} ),\quad (\underset{K}{\underbrace{\frac{1}{K},\cdots,\frac{1}{K}}})\Big\}
\end{equation*}
where $x\in (0,1)$.

Sections \ref{Sec4-1-1} and \ref{Sec4-1-2} concern the testing performance and Section \ref{Sec4-1-3} focuses on the estimation performance with RIRS.

\subsubsection{SBM}\label{Sec4-1-1}

When all rows of  $\bbPi$ are chosen from the pure node set $\text{PN}(K)$ and the degree heterogeneity matrix $\bbTheta=\rho\bbI_n$, the DCMM \eqref{0105.1} reduces to the SBM with the following mean matrix structure
\begin{equation}\label{0106.1}
\bbH=\rho\bbPi\bbB\bbPi^T,\quad \rho\in (0,1),\quad  \bbpi_i\in \text{PN}(K),\ i=1,\cdots, n.
\end{equation}

We generate 200 independent adjacency matrices each with $n=1000$ nodes  and $K$ equal-sized communities from the above SBM \eqref{0106.1}.
We set $\bbB=(B_{ij})_{K\times K}$ with $B_{ij}=s^{|i-j|}$, $i\neq j$ and $B_{ii}=(K+1-i)/K$.  The value of $\rho$ ranges from 0.04 to 0.9, with smaller $\rho$ corresponding to sparser network model.  For all values of  $K$, we choose $m=\sqrt n$ in calculating the RIRS test statistics $T_n$ and $\widetilde T_n$ for networks without and with selfloops, respectively.

The performance of RIRS is compared with the methods in \cite{L16}, where two versions of test -- one with and one  without bootstrap correction -- were proposed when the network is absent of selfloops. The empirical sizes and powers of both methods when $s=0.1$ are reported in Tables \ref{tab:1} and \ref{tab:2} for $K=2$ and $3$,  respectively. The corresponding computation times are reported in Table \ref{tab:time}.  We also test the performance of both methods with different values of $s$,  and the corresponding results are summarized in Table S.1 in the {\color{black} Supplementary} file.

From Tables \ref{tab:1} and \ref{tab:2}, we observe that  when $K=2$, the performance of RIRS is relatively robust  to the sparsity level $\rho$, with size close to the nominal level and power close to 1 in almost all settings.
 On contrary, the method in \cite{L16}  without bootstrap correction has much worse performance. It suffers from size distortion for smaller $\rho$ (sparser setting). This phenomenon becomes even more severe when $K=3$, where the sizes are close or even equal to one at all sparsity levels. With such distorted size, it is no longer meaningful to compare the power. Therefore we omit its power in Table \ref{tab:2}.   With bootstrap correction,  the method in \cite{L16} performs much better and is comparable to RIRS when $K=2$ except for the setting of $\rho=0.04$, where the size is severely distorted. When $K=3$,  both methods suffer from some size distortion when $\rho\leq 0.1$, where the problem is more severe for the method in \cite{L16}.   Comparing Table \ref{tab:1} with Table \ref{tab:2}, we see that the increased number of communities $K$ makes the performance of both methods worse. This is reasonable because the network size is fixed at $n=1000$.  So larger $K$ results in smaller size of each community.  From Table \ref{tab:time} we see that the computational cost of the bootstrap method in \cite{L16} is much higher than that of RIRS.  We also experimented with larger value of $K=5$ and the results are summarized in Table S.2 of the {\color{black}Supplementary} file to save space.

Finally, we present in Figure \ref{fig:sbm1} the histogram plots as well as the fitted density curves of our test statistics from 1000 repetitions when $K=2$, $s=0.1$, and $\rho=0.7$  under the null hypothesis. The standard normal density curves are also plotted as reference. It visually confirms that the asymptotic null distribution is standard normal.

\begin{table}[htbp]
	\centering
	{\small
		\caption{ Empirical size and power under SBM with $K=2$ and $s=0.1$.
		}
		\begin{tabular}{|c|cc|cc|cc|cc|}
			\hline
			& \multicolumn{6}{c|}{No selfloop} & \multicolumn{2}{c|}{Selfloops}  \\
			\hline
			& \multicolumn{2}{c|}{RIRS ($T_n$)}   &\multicolumn{2}{c|}{Lei (no bootstrap)} &\multicolumn{2}{c|}{Lei (bootstrap)} & \multicolumn{2}{c|}{RIRS ($\widetilde{T}_n$)}\\
			\hline
			$\rho$    & size  & $\underset{(K_0=1)}{\text{power}}$& size  & $\underset{(K_0=1)}{\text{power}}$ & size  & $\underset{(K_0=1)}{\text{power}}$& size  & $\underset{(K_0=1)}{\text{power}}$\\
			\hline
            0.04&0.055& 0.86&1&1&0.575& 1 & 0.140& 0.310\\
            0.07&0.060& 0.99&1&1& 0.055 &1 & 0.060& 0.605\\
            0.09&0.055 &1 &0.995&1&0.040 &1 & 0.085& 0.750\\
			0.1   & 0.025 & 1     & 0.995 & 1     &  0.035     &   1    & 0.085 & 0.815 \\
			0.3   & 0.025 & 1     & 0.24  & 1     &   0.02    &   1    & 0.06  & 1 \\
			0.5   & 0.045 & 1     & 0.07  & 1     &   0.025    &   1    & 0.065 & 1 \\
			0.7   & 0.065 & 1     & 0.1   & 1     &   0.055    &   1    & 0.05  & 1 \\
			0.9   & 0.04  & 1     & 0.045 & 1     &   0.065    & 1      & 0.075 & 1 \\
			\hline
		\end{tabular}%
		\label{tab:1}%
	}

\vspace{15pt}

	\centering
	{\small
		\caption{Empirical size and power under SBM with $K=3$ and $s=0.1$.
		}
\resizebox{\textwidth}{24mm}{
		\begin{tabular}{|c|ccc|c|ccc|ccc|}
			\hline
			& \multicolumn{7}{c|}{No selfloop} & \multicolumn{3}{c|}{Selfloops}  \\
			\hline
			& \multicolumn{3}{c|}{RIRS ($T_n$)}   &$\underset{\text{(no bootstrap)}}{\text{Lei}}$&\multicolumn{3}{c|}{$\underset{\text{( bootstrap)}}{\text{Lei}}$} & \multicolumn{3}{c|}{RIRS($\widetilde{T}_n$)}\\
			\hline
			$\rho$ &size & $\underset{(K_0=1)}{\text{power}}$ &$\underset{(K_0=2)}{\text{power}}$ &size &size & $\underset{(K_0=1)}{\text{power}}$ &$\underset{(K_0=2)}{\text{power}}$&size & $\underset{(K_0=1)}{\text{power}}$ &$\underset{(K_0=2)}{\text{power}}$ \\
			\hline
            0.04&0.225&0.985&0.235& 1&1&1&1&0.35&0.71&0.075\\
            0.07&0.27&1&     0.26& 1 &1&1&1&0.365&0.925&0.09\\
            0.09&0.225&1&    0.335&1 &0.985&1&1&0.21&1&0.12\\
			0.1   & 0.065 & 1     & 0.36  & 1          &   0.895    &    1   &   1    & 0.1   & 0.98  & 0.19 \\
			0.3   & 0.075 & 1     & 0.795 & 1          &    0.06   &    1   &     1  & 0.065 & 1     & 0.625 \\
			0.5   & 0.045 & 1     & 0.98  & 0.99       &   0.02    &     1  &    1   & 0.075 & 1     & 0.94 \\
			0.7   & 0.045 & 1     & 0.985 & 0.925      &      0.04  &   1    &   1    & 0.065 & 1     & 1 \\
			0.9   & 0.05  & 1     & 1     & 0.69       &     0.015  &   1    &    1   & 0.05  & 1     & 1 \\
			\hline
		\end{tabular}}%
		\label{tab:2}%
	}

\vspace{15pt}

  \centering
  \caption{Average computation time (in seconds) for test statistics in Table \ref{tab:1} and Table \ref{tab:5} in one replication under SBM with no selfloop, $K=2$ and $\rho=0.5$.}
    \begin{tabular}{|c|cc|cc|cc|}
    \hline
  & \multicolumn{2}{c|}{RIRS ($T_n$)}   &\multicolumn{2}{c|}{Lei (no bootstrap)} &\multicolumn{2}{c|}{Lei (bootstrap)}\\
  \hline
  & Size  &Estimation& Size  &Estimation& Size  & Estimation\\
  \hline
  Time& 0.504&0.906 & 0.432 & 2.88  &14.410&147.142 \\
  \hline
  \end{tabular}%
  \label{tab:time}%
\end{table}%

   \begin{figure}[!htb]
    \includegraphics[width=5.5in,  trim=0 0.8in 0 0.3in, clip]{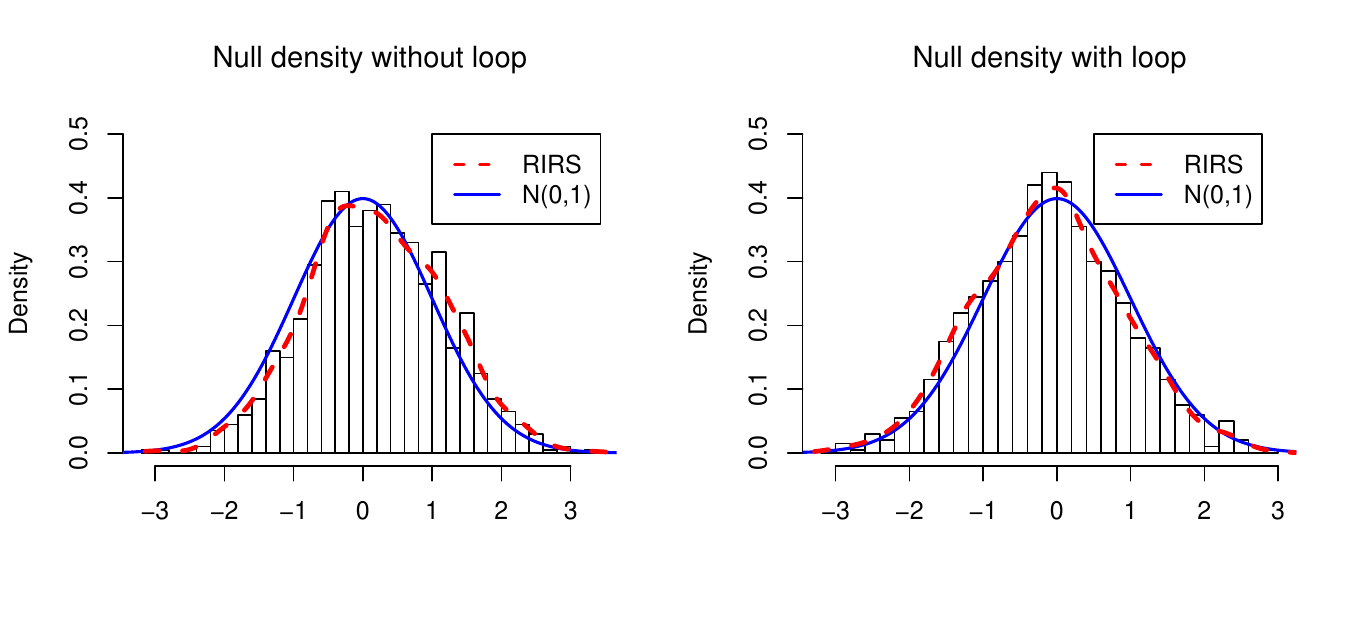}
    \caption{Histogram plots and the estimated densities (red curves) of RIRS test statistic when $K=2$ and $\rho=0.7$. Left: $T_n$ when no selfloop; Right: $\widetilde T_n$ when selfloops exist.}
    \label{fig:sbm1}
  \end{figure}

\subsubsection{DCMM}\label{Sec4-1-2}

Next consider the general DCMM model (\ref{0105.1}). The number of repetition is still 200.
We simulate the node degree parameters $\vartheta_j$'s independently  from the uniform distribution over $[0.5,1]$. The vectors $\bbpi_i$ are chosen from $\text{PN}(K)\cup\text{MM}(K,0.2)$, with $n_0$ pure nodes from each community  and $(n-Kn_0)/3$ nodes from each mixed membership probability mass vector in $\text{MM}(K,0.2)$. We select $n_0=0.35n$ when $K=2$ and $n_0=0.25n$ when $K=3$. The matrix $\bbB$ is chosen to be the same as in the SBM with $s=0.1$. The network size $n$ ranges from 800 to 2000. The empirical sizes and powers are summarized in Table \ref{tab:new1}.

Since \cite{L16} only considers SBM, the tests therein are no longer applicable in this setting. RIRS performs well and similarly to the SBM setting. Figure \ref{fig:dcmm1} presents the histogram plots as well as the fitted density curves of
RIRS under the null hypothesis from 1000 repetitions when $K = 3$ and $n=1500$. These results well justify our theoretical findings.

\begin{table}[!htb]
  \centering
  {\small
  \caption{Empirical size and power of RIRS under DCMM model.}
  \resizebox{\textwidth}{18mm}{
    \begin{tabular}{|c|cc|cc||ccc|ccc|}
        \hline
& \multicolumn{4}{c||}{$K=2$} & \multicolumn{6}{c|}{$K=3$} \\
    \hline
  & \multicolumn{2}{c|}{No Selfloop ($T_n$)} & \multicolumn{2}{c||}{Selfloop ($\widetilde T_n$)} &  \multicolumn{3}{c|}{No Selfloop ($T_n$)} & \multicolumn{3}{c|}{Selfloop ($\widetilde T_n$)}  \\
  \hline
$n$ &Size& $\underset{(K_0=1)}{\text{Power}}$ & Size      &$\underset{(K_0=1)}{\text{Power}}$ &Size & $\underset{(K_0=1)}{\text{Power}}$ &$\underset{(K_0=2)}{\text{Power}}$ &Size & $\underset{(K_0=1)}{\text{Power}}$ &$\underset{(K_0=2)}{\text{Power}}$  \\
\hline
    800   & 0.045 & 1     & 0.08  & 1 &0.05  & 1     & 0.58  & 0.08  & 1     & 0.845 \\
    1000  & 0.04  & 1     & 0.05  & 1 &0.025 & 1     & 0.68  & 0.06  & 1     & 0.92 \\
    1200  & 0.065 & 1     & 0.05  & 1 &0.045 & 1     & 0.77  & 0.07  & 1     & 0.92 \\
    1500  & 0.045 & 1     & 0.03  & 1 &0.075 & 1     & 0.9   & 0.055 & 1     & 0.98 \\
    1800  & 0.075 & 1     & 0.055 & 1 &0.045 & 1     & 0.98  & 0.065 & 1     & 0.995 \\
    2000  & 0.075 & 1     & 0.065 & 1 &0.05  & 1     & 0.965 & 0.045 & 1     & 1 \\
    \hline
    \end{tabular}}%
  \label{tab:new1}%
  }
\end{table}%

     \begin{figure}[!htp]
    \centering
    \includegraphics[width=5.2in,  trim=0 0.8in 0 0.3in, clip]{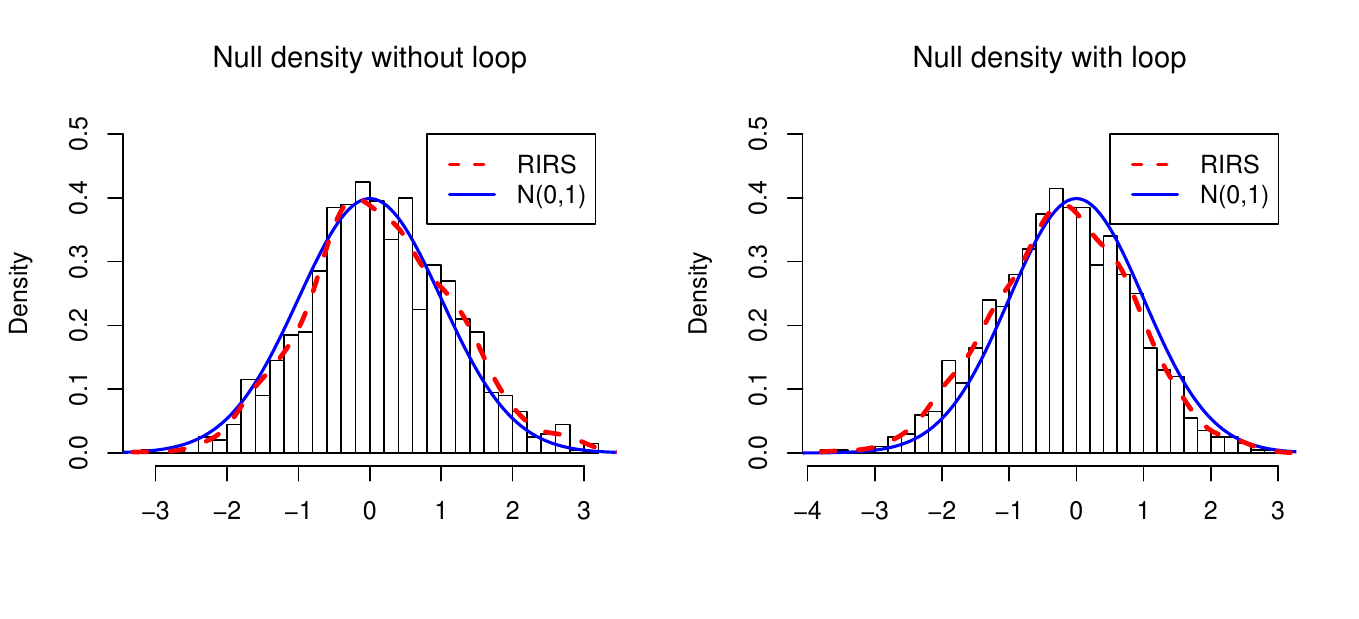}
    \caption{DCMM. Histogram plots and the estimated densities (red curves) of RIRS when $K=3$ and $n=1500$. Left: $T_n$ when no selfloop; Right: $\widetilde T_n$ when selfloops exist.}
    \label{fig:dcmm1}
  \end{figure}

\subsubsection{Estimating the Number of Communities}\label{Sec4-1-3}

We use the method discussed in Section \ref{sec: K-est} to estimate the number of communities $K$.
Since the approaches in \cite{L16} are not applicable to the DCMM model, we only compare the performance of RIRS with \cite{L16}  in SBM setting in the absence of selfloops. The proportions of correctly estimated $K$ are calculated over 200 replications and tabulated in Table \ref{tab:5} for SBM and in Table \ref{tab:6} for DCMM model.

Table \ref{tab:5} shows that RIRS is generally comparable to or is only slightly worse than Lei's method under the SBM when $\rho$ is large,   and is significantly better than the latter for very small $\rho$. While for DCMM model (Table \ref{tab:6}), RIRS can also estimate the number of communities with high accuracy. In particular, the estimation accuracy gets closer and closer to the expected value of 95\% as $n$ increases, which is consistent with our theory.

\begin{table}[htbp]
  \centering
  {\small
  \caption{Proportion of correctly estimated $K$ under SBM. }
   \resizebox{\textwidth}{22mm}{
    \begin{tabular}{|c|ccc|c||ccc|c|}
    \hline
    &\multicolumn{4}{c||}{$K=2$}&\multicolumn{4}{c|}{$K=3$}\\
    \hline
    &\multicolumn{3}{c|}{No Selfloop}&Selfloop&\multicolumn{3}{c|}{No Selfloop}& Selfloop\\
    \hline
    $\rho$&$\underset{T_n}{\text{RIRS}}$&$\underset{\text{(no bootstrap)}}{\text{Lei}}$&$\underset{\text{(bootstrap)}}{\text{Lei}}$&$\underset{\widetilde T_n}{\text{RIRS}}$&
    $\underset{T_n}{\text{RIRS}}$&$\underset{\text{(no bootstrap)}}{\text{Lei}}$&$\underset{\text{(bootstrap)}}{\text{Lei}}$&$\underset{\widetilde T_n}{\text{RIRS}}$\\
    \hline
    0.04&0.85&0&0.32&0.345&0.12&0&0&0.015\\
    0.07&0.945&0&0.98&0.59&0.2&0&0&0.04\\
    0.09&0.945&0&0.96&0.775&0.23&0&0.02&0.125\\
    0.1   & 0.93  & 0     & 0.97  & 0.815 & 0.285 & 0     & 0.165 & 0.105 \\
    0.3   & 0.94  & 0.795 & 0.955 & 0.96  & 0.745 & 0     & 0.935 & 0.645 \\
    0.5   & 0.945 & 0.925 & 0.925 & 0.95  & 0.9   & 0.005 & 0.98  & 0.895 \\
    0.7   & 0.97  & 0.915 & 0.945 & 0.96  & 0.955 & 0.065 & 0.995 & 0.955 \\
    0.9   & 0.94  & 0.94  & 0.935 & 0.955 & 0.93  & 0.275 & 0.975 & 0.945 \\
    \hline
    \end{tabular}}%
  \label{tab:5}%
  }
\end{table}%

\begin{table}[htbp]
  \centering
 {\small
  \caption{Proportion of correctly estimated $K$ under DCMM.}
   \resizebox{\textwidth}{12.4mm}{
    \begin{tabular}{|C{0.8cm}|C{0.66cm}C{0.66cm}C{0.66cm}C{0.66cm}C{0.66cm}C{0.66cm}||C{0.66cm}C{0.66cm}C{0.66cm}C{0.66cm}C{0.66cm}C{0.66cm}|}
    \hline
    &\multicolumn{6}{c||}{$K=2$}&\multicolumn{6}{c|}{$K=3$}\\
    \hline
    $n$     & 800   & 1000  & 1200  & 1500  & 1800  & 2000  &  800   & 1000  & 1200  & 1500  & 1800  & 2000 \\
    \hline
    &\multicolumn{12}{c|}{No Selfloop ($T_n$)}\\
    \hline
    RIRS   & 0.935 & 0.935 & 0.93  & 0.965 & 0.935 & 0.95  & 0.505 & 0.625 & 0.805 & 0.865 & 0.915 & 0.935 \\
    \hline
    &\multicolumn{12}{c|}{Selfloop ($\widetilde T_n$)}\\
    \hline
    RIRS   & 0.935 & 0.94  & 0.955 & 0.955 & 0.955 & 0.945 &  0.79  &  0.85  &  0.93  & 0.895 & 0.935 & 0.965 \\
    \hline
    \end{tabular}}%
  \label{tab:6}%
  }
\end{table}%

\subsection{Low rank data matrix}\label{Sec4-2}
RIRS can be applied to other low rank data matrices beyond the network model.
In this section, we generate $n\times n$ data matrix $\bbX$ from the following model
$$\bbX=\bbH+\bbW= \bbV\bbD\bbV^T+\bbW,$$
where the residual matrix $\bbW$ is symmetric with upper triangle entries (including the diagonal ones) i.i.d from uniform distribution over (-1,1). Let $\bbV=\frac{1}{\sqrt{2}}\left(
  \begin{array}{ccc}
    \bbV_1\\
    \bbV_2 \\
  \end{array}
\right) $, where $\bbV_1$ and $\bbV_2$ are  $n_1\times K$ and  $(n-n_1)\times K$ matrices respectively.
We randomly generate an $n_1\times n_1$ Wigner matrix and collect its $K$ eigenvectors corresponding to the largest $K$ eigenvalues to form $\bbV_1$. We set $\bbV_2=\frac{\sqrt K}{\sqrt{ n-n_1}}\bbPi$ with $\bbPi=(\bbpi_1,...,\bbpi_{n-n_1})^T$, where $\bbpi_i\in \text{PN}(K)$ and  the number of rows taking each distinct value from $\text{PN}(K)$ is the same. The diagonal matrix $\bbD=n\times \diag(K,K-1,...,1)$. We set $n_1=n/2$ and range the value of $n$ from 100 to 500. When $K=2$, the empirical sizes and powers as well as the proportions of correctly estimated $K$ over 500 repetitions are recorded in Table \ref{tab:lowrank}. It is seen that both $T_n$ and $\widetilde{T}_n$ performs well, with $\widetilde{T}_n$ having slightly higher power. This higher power further translates into better estimation accuracy (closer to 95\%) of estimated $K$.

\begin{table}[htbp]
  \centering
  {\footnotesize\caption{Empirical size and power,  and the proportion (Prop) of correctly estimated  $K$ over 500 replications. }
    \begin{tabular}{|l|ccccc||ccccc|}
    \hline
    &\multicolumn{5}{c||}{No Selfloop ($T_n$)}&\multicolumn{5}{c|}{Selfloop ($\widetilde T_n$)}\\
    \hline
    $n$&100&200&300&400&500&100&200&300&400&500\\
    \hline
  Size& 0.048 & 0.042 & 0.05  & 0.054 & 0.052 & 0.05  & 0.05  & 0.032 & 0.062 & 0.052 \\
 Power& 0.612 & 0.914 & 0.994 & 1     & 1     & 1     & 1     & 1     & 1     & 1 \\
  Prop&0.588 & 0.856 & 0.944 & 0.95  & 0.954 & 0.95  & 0.95  & 0.968 & 0.938 & 0.948 \\
 \hline
    \end{tabular}%
  \label{tab:lowrank}%
  }
\end{table}%

\section{Real data analysis} \label{Sec:real}
We consider a popularly studied network of political blogs assembled by \cite{Adamic05}.   The nodes are blogs over the period of two months before the 2004 U.S. Presidential Election. The edges are the web links between the blogs. These blogs have known political divisions and were labeled into two communities ($K=2$) by \cite{Adamic05} -- the liberal and conservative communities. This blog data has been frequently used in the literature, see \cite{Karrer2011}, \cite{zhao2012} and \cite{L16} among others. It is widely believed to follow a degree corrected block model.  %For the readers' convenience, we cite a graph (Figure \ref{fig:blogdata}) from \cite{Karrer2011}, which modeled the data using the degree corrected block model.
Following the literature, we ignore the directions and study only the largest connected component, which has $n=1222$ nodes.
Consider the following two hypothesis tests:
\begin{eqnarray*}
&&{\textbf {(HT1)}:}\quad H_0: K=1\ \text{ vs. } \ H_1: K>1.\\
&&{\textbf {(HT2)}:}\quad H_0: K=2\ \text{ vs. } \ H_1: K>2.
\end{eqnarray*}

\cite{L16} considered {\textbf{(HT2)}} and obtained test statistic values 1172.3 and 491.5, corresponding to the test without bootstrap and with bootstrap, respectively. Both are much larger than the critical value (about 1.454) from the Tracy-Widom distribution, and thus the null hypothesis in {\textbf{(HT2)}} was strongly rejected. This is not surprising because the testing procedure in \cite{L16} is based on the SBM. It is possible that the model is misspecified when applying the tests therein.

RIRS does not depend on any specific network model structure and is expected to be more robust to model misspecification. Since most of the diagonal entries of $\bbX$ are zero, we use the test statistic $T_n$. Noticing that the observed data matrix $\bbX$ is non-symmetric, we consider two simple transformations:
\begin{equation}\label{0107.1}
\text{Method 1}: \widetilde{\bbX}_1=\bbX+\bbX^T;\qquad
\text{Method 2}:\widetilde{\bbX}_2=\left(
\begin{array}{ccc}
\bbzero &\ \ \bbX\\
\bbX^T& \ \ \bbzero\\
\end{array}
\right)_{2n\times 2n}.
\end{equation}
The transformation in Method 2 is general and can be applied to even non-square data matrix $\bbX$.
After the transformations, $\rank(\E(\widetilde{\bbX}_1))=K$ and $\rank(\E(\widetilde{\bbX}_2))=2K$.
The results of applying $T_n$ to the two hypothesis test problems \textbf {(HT1)} and \textbf {(HT2)}, together with the estimated number of communities by the sequential testing procedure are reported in Table \ref{tab:realdata}. We can see that for both transformations, RIRS consistently estimated the number of communities to be 2, which is consistent with the common belief in the literature.

\begin{table}[h]
  \centering
  {\footnotesize\caption{Hypothesis testing and estimation results for the political blog data.}
    \begin{tabular}{|c|c|c||c|c|c|}
    \hline
    &\multicolumn{2}{c||}{Method 1}&\multicolumn{2}{c|}{Method 2}&\multirow{ 2}{*}{Decision}\\
    \cline{2-5}
    &Test Statistic&P-value&Test Statistic&P-value&\\
    \hline
   {\textbf{(HT1)}}&3.3527&0.0008&2.7131&0.0067&Reject $H_0$ in \textbf{(HT1)}\\
    \hline
    {\textbf{(HT2)}}& -1.2424&0.2141&-0.8936&0.3716& Accept $H_0$ in \textbf{(HT2)}\\
    \hline
    Estimate&\multicolumn{2}{c||}{2}&\multicolumn{2}{c|}{2}& $K=2$\\
    \hline
    \end{tabular}%
  \label{tab:realdata}%
  }
\end{table}%

\begin{appendix}
\section{Proof of the main results}\label{SecA}

We introduce a definition that will be used frequently in the proof.
\begin{definition}\label{def1}

Let $\zeta_n(i)$ and $\xi_n(i)$ be random (or deterministic) variables depend on $i$, $i=1,\ldots,n$. We say $\xi_n=O_{p_u}(\zeta_n)$ if  for any positive constants $D$ and $\ep$, there exists  some positive integer $n_0(D,\ep)$ depending only on  $D$ and $\ep$ such that for all $n\ge n_0(D,\ep)$ we have
	$$\mathbb{P}\left[\exists{1\le i \le n}, s.t. |\xi_n(i)|>( \log n)^{\ep}|\zeta_n(i)|\right]\le (\log \log n)^{-D}.$$
\end{definition}

In addition, to facilitate our proof presentation, we introduce some additional notations. Let $\tilde{\mathcal{K}}$ be the number of distinct nonzero eigenvalues of $\bbH=(h_{ij})$.  Denote by $\tilde{\mathrm{d}}_1, \cdots, \tilde{\mathrm{d}}_{\tilde{\mathcal K}}$ the distinct values in $\{d_1,\cdots, d_K\}$, sorting in decreasing magnitudes. In addition,  denote by $K_j$ the multiplicity of $\tilde{\mathrm{d}}_j$. That is, for each $j=1,\cdots, \tilde{\mathcal K}$, the cardinality of $\{1\leq l\leq K: d_l = \tilde{\mathrm{d}}_j \}$  is $K_j$ with $\sum_{j=1}^{\tilde{\mathcal K}}K_j = K$.

\subsection{Outline of The Proof} \label{Sec6}
Our Condition \ref{cond2} and main results Theorems \ref{thma}--\ref{thm2} allow for multiplicity in eigenvalues.  When proving these main results, we consider the case with and without multiplicity separately.
The proof of our main results highly depends on  the asymptotic expansions of the eigenvectors $
\hat \bv_k$ and eigenvalues $\hat d_k$.  Briefly speaking, Lemma 1 in Section 7.1 of the Supplementary material establishes the relationship between $\hat w_{ij}$ and $(\hat v_k(i), \hat d_k)$. This together with the asymptotic expansions of eigenvalues and eigenvectors gives us the asymptotic expansion of $\hat w_{ij}$. Substituting this asymptotic expansion into the proposed test statistics, we are able to prove our main theorems by
additional analysis and calculations.  In the case without eigenvalue multiplicity, the asymptotic expansions of eigenvalues and eigenvectors are established in Lemmas 2 and 3 in the Supplementary material. The corresponding results in the existence of multiplicity are established in Lemmas 4 and 5 in the Supplementary file. In the main paper we only provide the proof of Theorem \ref{thma}. All other proofs are relegated to the Supplement.

\subsection{Proof of Theorem \ref{thma}}
The result in Theorem \ref{thma} can be obtained by combing the following two results.
	\begin{align}
&\text{CLT}:	\qquad \frac{\sqrt m\sum_{i\neq j}\hat w_{ij}Y_{ij}}{\sqrt{2\sum_{i\neq j}\mathbb{E} w^2_{ij}}}\stackrel{d}{\rightarrow } N(0,1),\label{0307.1}\\
&\text{Consistency:}\qquad \frac{\sum_{i\neq j}\hat w_{ij}^2}{\sum_{i\neq j}\mathbb{E} w_{ij}^2}=1+o_{p}(1).\label{0307.4}
\end{align}
We next proceed with proving \eqref{0307.1} and \eqref{0307.4}.

We first verify the central limit theorem \eqref{0307.1}.
By S.44 and Corollary S.1 we have
\begin{align}\label{0113.7}
&\sum_{i\neq j}\hat w_{ij}Y_{ij}=\sum_{i\neq j}w_{ij}Y_{ij}-\sum_{k=1}^{K_0}\sum_{i\neq j}(\bbv_k(i)\bbv_k(j)Y_{ij})\bbv_k^T\bbW\bbv_k\\
&-\sum_{k=1}^{K_0}\sum_{i\neq j}Y_{ij}\frac{\bbe_i^T\bbW^2\bbv_k\bbv_k(j)+\bbe_j^T\bbW^2\bbv_k\bbv_k(i)}{t_k}
\non
&-\sum_{k=1}^{K_0}\sum_{i\neq j}Y_{ij}\frac{\bbr_k(i)\bbv_k(j)+\bbr_k(j)\bbv_k(i)}{t_k}+2\sum_{k=1}^{K_0}\frac{\bbv_k^T\mathbb{E}\bbW^2\bbv_k}{d_k}\sum_{i\neq j}(\bbv_k(i)\bbv_k(j)Y_{ij})\non
&-2\sum_{k=1}^{K_0}\frac{d_k}{t_k}\sum_{i\neq j}Y_{ij}\bbv_k(i)\bbs_{k,j}^T\bbW\bbv_k-\sum_{k=1}^{K_0}\sum_{i\neq j}\frac{d_k\bbe_i^T\bbW\bbv_k\bbe_j^T\bbW\bbv_kY_{ij}}{t_k^2}\non
&+\sum_{i\neq j}(Y_{ij}O_{p_u}(\frac{\alpha_n(\log n)^3}{n |d_{K_0}|}+\frac{1}{n^{3/2}})). \nonumber
\end{align}
Recall that $\mathbb{E}Y_{ij}=\frac{1}{m}$, $a_k=\sum_{i=1}^n\bbv_k(i)$ and $|a_k|\leq \sqrt{n}$.
Our aim is to bound all terms on the right hand side of equation \eqref{0113.7} except for the first term $\sum_{i\neq j}w_{ij}Y_{ij}$. We begin with
splitting the term
$$\sum_{i\neq j}(\bbv_k(i)\bbv_k(j)Y_{ij})\bbv_k^T\bbW\bbv_k$$
into two parts:
$$\frac{1}{m}\sum_{i\neq j}(\bbv_k(i)\bbv_k(j))\bbv_k^T\bbW\bbv_k \quad \text{and}\quad \sum_{i\neq j}(Y_{ij}-\frac{1}{m})(\bbv_k(i)\bbv_k(j))\bbv_k^T\bbW\bbv_k.$$
 For the first part, first note that since $|w_{ij}|$ is bounded, we have $|\bbv_k^T\mathbb E\bbW\bbv_k|\lesssim 1$. Then by  Corollary S.2 in the Supplementary material and Condition \ref{cond1}  we have
\begin{align}\label{0113.5}
&\frac{1}{m}\sum_{i\neq j}(\bbv_k(i)\bbv_k(j))\bbv_k^T\bbW\bbv_k = \frac{1}{m}(a_k^2-1)\bbv_k^T\bbW\bbv_k\\
&=\frac{1}{m}(a_k^2-1)\Big(\bbv_k^T(\bbW-\mathbb E \bbW)\bbv_k +\bbv_k^T\mathbb E \bbW\bbv_k \Big)=(a_k^2+1)(O_{p_u}(\frac{\alpha_n}{m\sqrt n})+O(\frac{1}{m})).\nonumber
\end{align}
For the second part, first note that $\mathbb{E}(\bbv_k^T\bbW\bbv_k)^2$ $=$ $\var(\bbv_k^T\bbW\bbv_k)$ $+$ $\mathbb{E}^2(\bbv_k^T\bbW\bbv_k)$ $\lesssim$ $\alpha_n^2/n + 1$. Since $Y_{ij}$, $i\le j$ are i.i.d with $\mathbb{E}Y_{ij}=\frac{1}{m}$, Corollary S.1 and  Condition \ref{cond1}  ensure that
\begin{align*}
&\var(\sum_{i\neq j}(Y_{ij}-\frac{1}{m})(\bbv_k(i)\bbv_k(j))\bbv_k^T\bbW\bbv_k)\\
 &= \mathbb E\Big[\var\Big(\sum_{i\neq j}(Y_{ij}-\frac{1}{m})(\bbv_k(i)\bbv_k(j))\bbv_k^T\bbW\bbv_k|\bbW\Big)\Big]\non
&\lesssim \frac{1}{m}\sum_{i\neq j}(\bbv_k(i)\bbv_k(j))^2\mathbb{E}(\bbv_k^T\bbW\bbv_k)^2\lesssim \frac{\alpha_n^2}{mn}+\frac{1}{m},
\end{align*}
then
\begin{equation}\label{add08.22}
\sum_{i\neq j}(Y_{ij}-\frac{1}{m})(\bbv_k(i)\bbv_k(j))\bbv_k^T\bbW\bbv_k=O_{p}(\frac{\alpha_n}{\sqrt{mn}}+\frac{1}{\sqrt m}).
\end{equation}
Therefore,
{\small\begin{equation}\label{add08.30}
\sum_{i\neq j}(\bbv_k(i)\bbv_k(j)Y_{ij})\bbv_k^T\bbW\bbv_k
=(a_k^2+1)O_{p_u}(\frac{\alpha_n}{m\sqrt n})+O(\frac{a_k^2+1}{m})+O_{p}(\frac{\alpha_n}{\sqrt{mn}})+O_{p}(\frac{1}{\sqrt{m}}).
\end{equation}}

Similar to \eqref{add08.30},  we get
$$\frac{\bbv_k^T\mathbb{E}\bbW^2\bbv_k}{d_k}\sum_{i\neq j}(\bbv_k(i)\bbv_k(j)Y_{ij})=O_{p_u}(\frac{a_k^2}{m}+\frac{1}{\sqrt m}),$$
and
\begin{equation}\label{add08.31}
\sum_{i\neq j}Y_{ij}\frac{\bbe_i^T\bbW^2\bbv_k\bbv_k(j)+\bbe_j^T\bbW^2\bbv_k\bbv_k(i)}{t_k}=O_{p_u}(\frac{|a_k|\sqrt n}{m}+\frac{1}{\sqrt m}).
\end{equation}

Next we split the term
$\sum_{i\neq j}Y_{ij}\frac{d_k\bbs_{k,j}^T\bbW\bbv_k\bbv_k(i)}{t_k}$
into the following two parts
$$\frac{1}{m}\sum_{i\neq j}\frac{d_k\bbs_{k,j}^T\bbW\bbv_k\bbv_k(i)}{t_k} \quad \text{and}\quad \sum_{i\neq j}(Y_{ij}-\frac{1}{m})\frac{d_k\bbs_{k,j}^T\bbW\bbv_k\bbv_k(i)}{t_k}.$$
By (S.85), we have
\begin{align*}
&\|\mathcal{R}(\mathbf{1},\bbV_{-k},t_k)\Big((\bbD_{-k})^{-1}+\mathcal{R}(\bbV_{-k},\bbV_{-k},t_k)\Big)^{-1}\bbV_{-k}^T\|\\
&=\|\sum_i\mathcal{R}(\bbe_i,\bbV_{-k},t_k)\Big((\bbD_{-k})^{-1}+\mathcal{R}(\bbV_{-k},\bbV_{-k},t_k)\Big)^{-1}\bbV_{-k}^T\|=O_{p_u}(\sqrt n).
\end{align*}
In light of (S.46), (S.85), Corollary S.1  in the Supplementary material and Condition \ref{cond4}, the following three results hold:
\begin{align*}
&\frac{d_k(\bbs_{k}-\mathbf{1})^T \bbW\bbv_ka_k}{t_k}\non
&=-\frac{d_ka_k\mathcal{R}(\mathbf{1},\bbV_{-k},t_k)\Big((\bbD_{-k})^{-1}+\mathcal{R}(\bbV_{-k},\bbV_{-k},t_k)\Big)^{-1}\bbV_{-k}^T\bbW\bbv_k}{t_k}\\
&\quad-\frac{d_ka_k^2\bbv_k^T\bbW\bbv_k}{t_k}=O_{p_u}(\alpha_n+|a_k|\sqrt n)+O_{p_u}(|a_k|\alpha_n),
\end{align*}
\begin{align*}
\frac{1}{m}\sum_{i\neq j}\frac{d_k\bbs_{k,j}^T\bbW\bbv_k\bbv_k(i)}{t_k}&=
\frac{1}{m}a_k\frac{d_k\mathbf{1}^T(\bbW-\mathbb{E}\bbW)\bbv_k}{t_k}+O_{p_u}(\frac{(|a_k|+1)\alpha_n}{m})\non
&=O_{p_u}(\frac{(|a_k|+1)\alpha_n}{m})
\end{align*}
\begin{equation*}
\text{and}\qquad\sum_{i\neq j}(Y_{ij}-\frac{1}{m})\frac{d_k\bbs_{k,j}^T\bbW\bbv_k\bbv_k(i)}{t_k}=O_{p}(\frac{\alpha_n}{\sqrt m}),
\end{equation*}
where the calculation of the variance of the second part $\sum\limits_{i\neq j}(Y_{ij}-\frac{1}{m})\frac{d_k\bbs_{k,j}^T\bbW\bbv_k\bbv_k(i)}{t_k}$ is similar to that of \eqref{add08.22}.
Therefore,
\begin{equation}\label{add08.32}
\sum_{i\neq j}Y_{ij}\frac{d_k\bbs_{k,j}^T\bbW\bbv_k\bbv_k(i)}{t_k}=O_{p_u}(\frac{(|a_k|+1)\alpha_n+|a_k|\sqrt n}{m})
+O_{p}(\frac{\alpha_n}{\sqrt m}).
\end{equation}

For the term $\sum_{i\neq j}Y_{ij}\frac{\bbr_k(i)\bbv_k(j)+\bbr_k(j)\bbv_k(i)}{t_k}$, we write
\begin{align}\label{0114.2}
\sum_{i\neq j}Y_{ij}\frac{\bbr_k(i)\bbv_k(j)+\bbr_k(j)\bbv_k(i)}{t_k}&=\frac{1}{m}\sum_{i\neq j}\frac{\bbr_k(i)\bbv_k(j)+\bbr_k(j)\bbv_k(i)}{t_k}\non
&\quad+\sum_{i\neq j}(Y_{ij}-\frac{1}{m})\frac{\bbr_k(i)\bbv_k(j)+\bbr_k(j)\bbv_k(i)}{t_k}.
\end{align}
It follows from  Conditions \ref{cond2}--\ref{cond4} and \ref{cond6}, $\frac{t_k}{d_k}\rightarrow 1$ in Section \ref{Sec3}
 and Corollary S.1 that
\begin{equation*}
|\mathbf{1}^T\bbr_k|=|\mathbf{1}^T\bbV_{-k}(t_k\bbD_{-k}^{-1}-\bbI)^{-1}\bbV_{-k}^T\mathbb{E}\bbW^2\bbv_k|=O_{p_u}( \sqrt n\alpha_n^2),
\end{equation*}
\begin{equation*}
|\bbr_k(i)|=|\mathbf{e}^T_i\bbV_{-k}(t_k\bbD_{-k}^{-1}-\bbI)^{-1}\bbV_{-k}^T\mathbb{E}\bbW^2\bbv_k|=O_{p_u}( \frac{\alpha_n^2}{\sqrt n}),
\end{equation*}
and thus

$$\frac{1}{m}\sum_{i\neq j}\frac{\bbr_k(i)\bbv_k(j)+\bbr_k(j)\bbv_k(i)}{t_k}=\frac{2a_k\mathbf{1}^T\bbr_k}{t_km}-\frac{2}{m}\sum_{i=1}^n\frac{\bbr_k(i)\bbv_k(i)}{t_k}
=O_{p_u}(\frac{|a_k|\sqrt n+1}{m}).$$
As for \eqref{add08.22}, calculating the variance of the second term on the right hand side of (\ref{0114.2}) yields
$$\sum_{i\neq j}(Y_{ij}-\frac{1}{m})\frac{\bbr_k(i)\bbv_k(j)+\bbr_k(j)\bbv_k(i)}{t_k}=O_{p_u}(\frac{1}{\sqrt m}).$$
Therefore,
\begin{equation}\label{add08.33}
\sum_{i\neq j}Y_{ij}\frac{\bbr_k(i)\bbv_k(j)+\bbr_k(j)\bbv_k(i)}{t_k}
=O_{p_u}(\frac{|a_k|\sqrt n}{m})+O_{p_u}(\frac{1}{\sqrt m}).
\end{equation}

Now for the term $\sum_{i\neq j}\frac{d_k\bbe_i^T\bbW\bbv_k\bbe_j^T\bbW\bbv_kY_{ij}}{t_k^2}$, similarly we write
{\small
$$\sum_{i\neq j}\frac{d_k\bbe_i^T\bbW\bbv_k\bbe_j^T\bbW\bbv_kY_{ij}}{t_k^2}=\sum_{i\neq j}\frac{d_k\bbe_i^T\bbW\bbv_k\bbe_j^T\bbW\bbv_k}{mt_k^2}+\sum_{i\neq j}\frac{d_k\bbe_i^T\bbW\bbv_k\bbe_j^T\bbW\bbv_k(Y_{ij}-\frac{1}{m})}{t_k^2}.$$
}
It follows from {Corollary S.1 and} Theorem S.1 that the first part has order
\begin{eqnarray*}
\sum_{i\neq j}\frac{d_k\bbe_i^T\bbW\bbv_k\bbe_j^T\bbW\bbv_k}{mt_k^2}
&=&\frac{d_k\bbone^T\bbW\bbv_k\bbone^T\bbW\bbv_k}{mt_k^2}-
\sum_{i=1}^n\frac{d_k(\bbe_i^T\bbW\bbv_k)^2}{mt_k^2}\non
&=&O_{p_u}(\frac{\alpha_n^2(\log n)^2}{m|d_k|})=O_{p_u}(\frac{(\log n)^2}{m}).
\end{eqnarray*}
Moreover, calculating the variance, we have
{\small\begin{eqnarray*}
&&\var\Big(\sum_{i\neq j}\frac{d_k\bbe_i^T\bbW\bbv_k\bbe_j^T\bbW\bbv_k(Y_{ij}-\frac{1}{m})}{t_k^2}\Big)\non
&=&\mathbb E\Big(\left[\var\Big(\sum_{i\neq j}\frac{d_k\bbe_i^T\bbW\bbv_k\bbe_j^T\bbW\bbv_k(Y_{ij}-\frac{1}{m})}{t_k^2}\Big)\Big|\bbW\right]\Big)\non
&\lesssim& \frac{\sum_{i\neq j}\mathbb{E}(\bbe_i^T\bbW\bbv_k\bbe_j^T\bbW\bbv_k)^2}{md^2_k}\leq \frac{\sum_{i\neq j}\sqrt{\mathbb{E}(\bbe_i^T\bbW\bbv_k)^4\mathbb{E}(\bbe_j^T\bbW\bbv_k)^4}}{md^2_k}\nonumber
\end{eqnarray*}
}
{\small\begin{eqnarray*}
&\lesssim& \frac{\sum\limits_{i\neq j}\sqrt{[\mathbb{E}(\bbe_i^T\bbW\bbv_k-\mathbb{E}\bbe_i^T\bbW\bbv_k)^4+(\mathbb{E}\bbe_i^T\bbW\bbv_k)^4][\mathbb{E}(\bbe_j^T\bbW\bbv_k-\mathbb{E}\bbe_j^T\bbW\bbv_k)^4+(\mathbb{E}\bbe_j^T\bbW\bbv_k)^4]}}{md^2_k}\non
&\lesssim& \frac{\alpha_n^4}{md_k^2}=O_{p_u}(\frac{1}{m}).
\end{eqnarray*}
}
Therefore
\begin{equation}\label{add08.34}
\sum_{i\neq j}\frac{d_k\bbe_i^T\bbW\bbv_k\bbe_j^T\bbW\bbv_kY_{ij}}{t_k^2}=O_{p_u}(\frac{(\log n)^2}{m})+O_{p_u}(\frac{1}{\sqrt{m}}).
\end{equation}

Finally, consider the residual term
\begin{eqnarray*}
&&\sum_{i\neq j}(Y_{ij}O_{p_u}(\frac{\alpha_n(\log n)^3}{n |d_{K_0}|}+\frac{1}{n^{3/2}}))\non
&=&\sum_{i\neq j}O_{p_u}(\frac{\alpha_n(\log n)^3}{n |d_{K_0}|}+\frac{1}{n^{3/2}})(Y_{ij}-\frac{1}{m})+\frac{1}{m}\sum_{i\neq j}O_{p_u}(\frac{\alpha_n(\log n)^3}{n |d_{K_0}|}+\frac{1}{n^{3/2}}).
\end{eqnarray*}
Note that $Y_{ij}$ is independent of $O_{p_u}(\frac{\alpha_n(\log n)^3}{n |d_{K_0}|}+\frac{1}{n^{3/2}})$. Calculating the variance of $\sum\limits_{i\neq j}O_{p_u}(\frac{\alpha_n(\log n)^3}{n |d_{K_0}|}+\frac{1}{n^{3/2}})(Y_{ij}-\frac{1}{m})$ gives us
\begin{eqnarray*}
\sum_{i\neq j}O_{p_u}(\frac{\alpha_n(\log n)^3}{n |d_{K_0}|}+\frac{1}{n^{3/2}})(Y_{ij}-\frac{1}{m})=O_p(\frac{1}{\sqrt m})\times O_{p_u}(\frac{\alpha_n(\log n)^3}{|d_{K_0}|}+\frac{1}{n^{1/2}}).
\end{eqnarray*}
The ``mean'' of the residual term should be
$$\frac{1}{m}\sum_{i\neq j}O_{p_u}(\frac{\alpha_n(\log n)^3}{n |d_{K_0}|}+\frac{1}{n^{3/2}})=O_{p_u}(\frac{n\alpha_n(\log n)^3}{m |d_{K_0}|}+\frac{\sqrt n}{m}).$$
Therefore we have
\begin{eqnarray}\label{add08.35}
&&\sum_{i\neq j}(Y_{ij}O_{p_u}(\frac{\alpha_n(\log n)^3}{n |d_{K_0}|}+\frac{1}{n^{3/2}}))\\
&=&O_{p_u}(\frac{n\alpha_n(\log n)^3}{m |d_{K_0}|}+\frac{\sqrt n}{m})
+O_p(\frac{1}{\sqrt m})\times O_{p_u}(\frac{\alpha_n(\log n)^3}{|d_{K_0}|}+\frac{1}{\sqrt n}).\nonumber
\end{eqnarray}

So far we have found the orders of all other terms on the right hand side of equation \eqref{0113.7} except for $\sum_{i\neq j}w_{ij}Y_{ij}$. Note that
\begin{equation}\label{0113.6}
\var(\sum_{i\neq j}w_{ij}Y_{ij})=\frac{2}{m}\sum_{i\neq j}\mathbb{E}w^2_{ij}.
\end{equation}
According to the orders (\ref{add08.30}), \eqref{add08.31}, \eqref{add08.32}, \eqref{add08.33}, \eqref{add08.34} and \eqref{add08.35}, we can conclude that
as long as
$$\frac{2}{m}\sum_{i\neq j}\mathbb{E}w^2_{ij}\ge (\log n)^{\ep_1}(\frac{\sum_{k=1}^{K_0}a_k^2n}{m^2}+\frac{(\alpha_n\log n)^2}{m}+\frac{n}{m^2}+\frac{n^2\alpha_n^2(\log n)^6}{m^2d_{K_0}^2}),$$
 the term $\sum_{i\neq j}w_{ij}Y_{ij}$ dominates all other terms on the right hand side of (\ref{0113.7}). Moreover, by the condition $\sum_{i\neq j}\mathbb{E}w^2_{ij}\gg m$, the fact $\mathbb{E}Y_{ij}^4\lesssim 1/m$ and the independence between $Y_{ij}$ and $w_{ij}$ we have
$$\frac{m^2}{(\sum_{i\neq j}\mathbb{E}w^2_{ij})^2}\sum_{i\neq j}\mathbb{E}w^4_{ij}Y_{ij}^4\lesssim \frac{m\sum_{i\neq j}\mathbb{E}w^4_{ij}}{(\sum_{i\neq j}\mathbb{E}w^2_{ij})^2} \lesssim \frac{m}{(\sum_{i\neq j}\mathbb{E}w^2_{ij})}\rightarrow 0.$$
Therefore, the central limit theorem (\ref{0307.1}) holds by Lyapunov CLT.

We now show the consistency of $\sum_{i\neq j}\hat w_{ij}^2$ in \eqref{0307.4}.
By (S.31), we have
\begin{align*}
\hat w_{ij}^2&=w_{ij}^2+2w_{ij}O_{p_u}(\frac{(\alpha_n\log n)^2}{n|d_{K_0}|}+\frac{1}{n})-2\sum_{k=1}^{K_0}w_{ij}\frac{d_k(\bbe_i^T\bbW\bbv_k\bbv_k(j)+\bbe_j^T\bbW\bbv_k\bbv_k(i))}{t_k}\\
&\quad+O_{p_u}(\frac{(\alpha_n\log n)^4}{n^2|d_{K_0}|^2}+\frac{1}{n^2})+(\sum_{k=1}^{K_0}\frac{d_k(\bbe_i^T\bbW\bbv_k\bbv_k(j)+\bbe_j^T\bbW\bbv_k\bbv_k(i))}{t_k})^2\non
&\quad-2O_{p_u}(\frac{(\alpha_n\log n)^2}{n|d_{K_0}|}+\frac{1}{n})\sum_{k=1}^{K_0}\frac{d_k(\bbe_i^T\bbW\bbv_k\bbv_k(j)+\bbe_j^T\bbW\bbv_k\bbv_k(i))}{t_k}\nonumber
\end{align*}
and
\begin{align}\label{0113.3}
&\sum_{i\neq j}\hat w_{ij}^2=\sum_{i\neq j}w_{ij}^2+2\sum_{i\neq j}(w_{ij}O_{p_u}(\frac{(\alpha_n\log n)^2}{n|d_{K_0}|}+\frac{1}{n}))\\
&\quad-2\sum_{k=1}^{K_0}\sum_{i\neq j}(w_{ij}\frac{d_k(\bbe_i^T\bbW\bbv_k\bbv_k(j)+\bbe_j^T\bbW\bbv_k\bbv_k(i))}{t_k})\non
&\quad+O_{p_u}(\frac{\alpha_n^2(\log n)^4}{|d_{K_0}|}+1)+\sum_{i\neq j}(\sum_{k=1}^{K_0}\frac{d_k(\bbe_i^T\bbW\bbv_k\bbv_k(j)+\bbe_j^T\bbW\bbv_k\bbv_k(i))}{t_k})^2\non
&\quad-2\sum_{i\neq j}O_{p_u}(\frac{(\alpha_n\log n)^2}{n|d_{K_0}|}+\frac{1}{n})\sum_{k=1}^{K_0}\frac{d_k(\bbe_i^T\bbW\bbv_k\bbv_k(j)+\bbe_j^T\bbW\bbv_k\bbv_k(i))}{t_k}.\nonumber
\end{align}
Combing the fact
$\var(\sum_{i\neq j}w_{ij}^2)\le \sum_{i\neq j}\mathbb{E}w_{ij}^4\lesssim \sum_{i\neq j}\mathbb{E}w_{ij}^2$
with Conditions \ref{cond3} and \ref{cond5}
\begin{equation}\label{0113.2}
\sum_{i\neq j}\mathbb{E}w_{ij}^2\gg (\log n)^{2+\ep_1},
\end{equation}
we have
\begin{equation}\label{add08.36}
\frac{\sum_{i\neq j}^nw_{ij}^2}{\sum_{i\neq j}^n\mathbb{E}w_{ij}^2}=1+O_{p_u}(\frac{1}{(\log n)^{\ep_1/2}}).
\end{equation}
Then to prove \eqref{0307.4}, it suffices to show
\begin{equation}\label{add08.37}
\frac{\sum_{i\neq j}\hat w_{ij}^2}{\sum_{i\neq j}w_{ij}^2}=1+O_{p_u}(\frac{1}{(\log n)^{\ep_1/2}}).
\end{equation}
We now check the other terms on the right hand side of \eqref{0113.3} to verify \eqref{add08.37}.
First of all,
\begin{align*}
|\sum_{i\neq j}w_{ij}O_{p_u}(\frac{(\alpha_n\log n)^2}{n|d_{K_0}|}+\frac{1}{n})|&\le \bigg|O_{p_u}(\frac{(\alpha_n\log n)^2}{|d_{K_0}|}+1)\times \sqrt{\sum_{i\neq j}w_{ij}^2}\bigg|\non
&=O_{p_u}(\alpha_n(\log n)\sqrt{\sum_{i\neq j}\mathbb{E}w_{ij}^2}).
\end{align*}
Condition \ref{cond5} further implies that
\begin{equation}\label{add08.40}
|\sum_{i\neq j}w_{ij}O_{p_u}(\frac{(\alpha_n\log n)^2}{n|d_{K_0}|}+\frac{1}{n})|=(\sum_{i\neq j}\mathbb{E}w_{ij}^2)\times O_{p_u}(\frac{1}{(\log n)^{\ep_1/2}}).
\end{equation}
Now consider the term $\sum_{i\neq j}w_{ij}\frac{d_k(\bbe_i^T\bbW\bbv_k\bbv_k(j)+\bbe_j^T\bbW\bbv_k\bbv_k(i))}{t_k}$. We will only provide detail for proving  $\sum_{i\neq j}w_{ij}\frac{d_k\bbe_i^T\bbW\bbv_k\bbv_k(j)}{t_k}$ because the other part can be proved similarly. Write
\begin{eqnarray*}
\sum_{i\neq j}w_{ij}\frac{d_k\bbe_i^T\bbW\bbv_k\bbv_k(j)}{t_k}=\sum_{i\neq j}\frac{d_kw_{ij}^2\bbv^2_k(j)}{t_k}+\sum_{i\neq j, l\neq j}\frac{d_kw_{ij}w_{il}\bbv_k(j)\bbv_k(l)}{t_k}.
\end{eqnarray*}
Direct calculations yield
$$\E\bigg|\sum_{i\neq j}\frac{d_k w_{ij}^2\bbv^2_k(j)}{t_k}\bigg|\lesssim\frac{1}{n}\sum_{i\neq j}\mathbb{E}w_{ij}^2,\quad \sum_{i\neq j, l\neq j}\mathbb{E}\frac{d_kw_{ij}w_{il}\bbv_k(j)\bbv_k(l)}{t_k}=0,\quad \text{and}$$
$$ \var(\sum_{i\neq j, l\neq j}\frac{d_kw_{ij}w_{il}\bbv_k(j)\bbv_k(l)}{t_k})\lesssim\sum_{i\neq j, l\neq j}\mathbb E w^2_{ij}w^2_{il}\bbv^2_k(j)\bbv^2_k(l)\lesssim\frac{\alpha_n^2}{n^2}\sum_{i\neq j}\mathbb{E}w_{ij}^2.$$
Thus, $\sum_{i\neq j}w_{ij}\frac{d_k\bbe_i^T\bbW\bbv_k\bbv_k(j)}{t_k}=O_{p_u}(\frac{1}{\sqrt n})\times \sum_{i\neq j}\mathbb{E}w_{ij}^2$. And consequently,
\begin{equation}\label{add08.41}
\sum_{i\neq j}w_{ij}\frac{d_k(\bbe_i^T\bbW\bbv_k\bbv_k(j)+\bbe_j^T\bbW\bbv_k\bbv_k(i))}{t_k}=O_{p_u}(\frac{1}{\sqrt n})\times \sum_{i\neq j}\mathbb{E}w_{ij}^2.
\end{equation}
Next by Theorem S.1 and Condition \ref{cond5}, we have
\begin{align}\label{add08.42}
&\sum_{i\neq j}O_{p_u}(\frac{(\alpha_n\log n)^2}{n|d_{K_0}|}+\frac{1}{n})\sum_{k=1}^{K_0}\frac{d_k(\bbe_i^T\bbW\bbv_k\bbv_k(j)+\bbe_j^T\bbW\bbv_k\bbv_k(i))}{t_k}\non
&=O_{p_u}(\frac{(\alpha_n\log n)^3}{|d_{K_0}|}+\alpha_n\log n)=O_{p_u}(\alpha_n^2\log^2n)=O_{p_u}(\frac{1}{(\log n)^{\ep_1/2}}\sum_{i\neq j}\mathbb{E}w_{ij}^2).
\end{align}
Finally, similar to (\ref{add08.42}), it holds by Condition \ref{cond5} that
\begin{align}\label{1106.4}
&\sum_{i\neq j}\Big(\sum_{k=1}^{K_0}\frac{d_k(\bbe_i^T\bbW\bbv_k\bbv_k(j)+\bbe_j^T\bbW\bbv_k\bbv_k(i))}{t_k}\Big)^2\\
&\lesssim \sum_{k=1}^{K_0}\sum_{i\neq j}\frac{d_k^2(\bbe_i^T\bbW\bbv_k\bbv_k(j)+\bbe_j^T\bbW\bbv_k\bbv_k(i))^2}{t_k^2}\non
&=O_{p_u}(\alpha_n^2\log^2n)=O_{p_u}(\frac{1}{(\log n)^{\ep_1/2}}\sum_{i\neq j}\mathbb{E}w_{ij}^2).\nonumber
\end{align}
Substituting the arguments \eqref{add08.36}, \eqref{add08.40}, \eqref{add08.41}, \eqref{add08.42} and \eqref{1106.4} into equation \eqref{0113.3}, we complete the proof of \eqref{add08.37}. Thus, \eqref{0307.4} is proved and the results in the theorem follow automatically.

\end{appendix}

\begin{funding}
Qing Yang was supported by National Natural Science Foundation of China (No. 12101585). Xiao Han was partly supported by NSF of China (No.12001518). Yingying Fan was supported by NSF grant DMS-2052964.
\end{funding}

\bibliographystyle{imsart-number}
\bibliography{references}

\end{document}